\documentclass{amsart}
\usepackage{amssymb,amsmath,amsthm}
\usepackage{graphicx}
\title[Integrable Complex Structures on Twistor Spaces]{Integrable Complex Structures on Twistor Spaces}
\author{Steven Gindi}
\newtheorem{thm}{Theorem}[section]
\newtheorem{lemma}[thm]{Lemma}
\newtheorem{prop}[thm]{Proposition}
\newtheorem{cor}[thm]{Corollary}
\newtheorem{rmk}[thm]{Remark} 
 
\newtheorem{nota}[thm]{Notation}
\theoremstyle{definition} \newtheorem{example}[thm]{Example}
\theoremstyle{definition}  \newtheorem{defi}[thm]{Definition}
\numberwithin{equation}{section}
\begin{document}
\begin{large}

\begin{abstract}
We introduce integrable complex structures on twistor spaces fibered over complex manifolds. We then show, in particular, that the twistor spaces associated with generalized Kahler, SKT and strong HKT manifolds all naturally admit complex structures. Moreover, in the strong HKT case we construct a metric and three compatible complex structures on the twistor space that have equal torsions. 
\end{abstract}

\maketitle
\tableofcontents

\section{Introduction}
In the 1970's, Atiyah, Hitchin and Singer introduced a tautological almost complex structure on a certain twistor space that along with its generalizations have had, until today, a major impact on differential and complex geometry \cite{Ahs1,Pen1}. The twistor space that they considered was $\mathcal{T}^{+}(TM)$, the bundle of complex structures fibered over an oriented Riemannian four manifold that are compatible with the metric and orientation; the almost complex structure was $J^{\nabla}_{taut}$, where $\nabla$ is the Levi Civita connection, and is defined in Section \ref{SecDCS}. The importance of $J^{\nabla}_{taut}$ was found to lie in the times when it was integrable, which was when the four manifold was anti-selfdual, and one of its many applications was the construction of instantons on $S^{4}$ \cite{Adhm1}.

Given its success in four dimensions,  $J^{\nabla}_{taut}$ was generalized in \cite{Ber1,Raw1} to the twistor space $\mathcal{C}(TM)= \{J \in EndTM | \ J^{2}=-1 \}$, where $M$ is any even dimensional manifold and $\nabla$ is any connection. Its integrability conditions were then explored with the hope that $J^{\nabla}_{taut}$ would again lead to major results about the base manifold. However, it was found that these conditions imposed severe restrictions on the curvature of the connection and in almost all cases $J^{\nabla}_{taut}$ was not integrable, thus limiting its applications in differential geometry.   

Of course this did not prevent mathematicians from taking advantage of the rare times when it was integrable on either $\mathcal{C}(TM)$ or on submanifolds within, and applying $J^{\nabla}_{taut}$ to advance, for example, the theory of harmonic mappings, integrable systems and hyperkahler geometry.
With all of its successes, however, the fact remains that the rarity of the integrability of $J^{\nabla}_{taut}$ has greatly hindered its use in deriving results about the geometry of the base manifold $M$ in higher dimensions. And it is natural to wonder whether there exist other almost  complex structures on twistor spaces whose integrability conditions are more easily satisfied---especially in every dimension---and at the same time  can  be used to derive results about the base manifold. 

The purpose of this paper is to demonstrate that such almost complex structures do indeed exist if we assume that $M$ is itself equipped with a complex structure $I$. Whereas $\mathcal{J}_{taut}$, which will stand for $J^{\nabla}_{taut}$ for an unspecified connection, can only be defined on twistor spaces 
 that are associated to $TM$, the almost complex structures that we introduce in Section \ref{SecDCS} are defined on more general twistor spaces that are associated to any even rank real vector bundle. Denoting such a  bundle by $E$, in that section, we define the almost complex structure $\mathcal{J}^{(\nabla,I)}$ on $\mathcal{C}(E) = \{ J \in EndE | J^{2}=-1\}$; it  depends on a choice of a connection $\nabla$ on $E$, similar to the definition of $\mathcal{J}_{taut}$. However, unlike $\mathcal{J}_{taut}$, the conditions on the connection $\nabla$ for $\mathcal{J}^{(\nabla,I)}$ to be integrable are easily fulfilled. By computing its Nijenhuis tensor,  we prove in Theorem \ref{Thm1c},  that $\mathcal{J}^{(\nabla,I)}$ on $\mathcal{C}(E)$ is automatically integrable if the curvature of $\nabla$, $R^{\nabla}$, is (1,1) with respect to $I$, i.e., $R^{\nabla}(I\cdot,I\cdot)=R^{\nabla}(\cdot,\cdot)$. Moreover if $g$ is a fiberwise metric on $E$ and $\nabla g=0$ then $\mathcal{J}^{(\nabla,I)}$ on $\mathcal{T}(E,g)= \{ J \in \mathcal{C}(E) | \ g(J\cdot,J\cdot)=g(\cdot,\cdot) \}$ is integrable if and only if $R^{\nabla}$ is (1,1) (Theorem \ref{Prop1c}). Under these conditions, the projection map $\pi: (\mathcal{C}(E),\mathcal{J}^{(\nabla,I)}) \longrightarrow (M,I)$ becomes a holomorphic submersion, a property that would never hold if we were to replace  $\mathcal{J}^{(\nabla,I)}$ by $\mathcal{J}_{taut}$.

While we will use $(\mathcal{C}(E),\mathcal{J}^{(\nabla,I)})$  to derive results about the base manifold $(M,I)$ in \cite{Gindi1}, the focus of our present paper is to describe various examples of vector bundles that admit connections with (1,1) curvature and to study the resulting holomorphic structures on the twistor spaces.  For instance, in Section \ref{SecDOP} we demonstrate that a general holomorphic Hermitian  bundle, $(E,g,J)$, admits many such connections. The Chern connection is of course an example, but as we show, $\overline{\partial}$ closed sections $D \in \Gamma(T^{*0,1}\otimes \wedge^{2}E^{*1,0})$ and $\alpha \in \Gamma(T^{*0,1}\otimes \wedge^{2}E^{1,0})$ can also be used to define connections on $E$ with (1,1) curvature. In the case when $E=TM$, $D$ is a $\overline{\partial}$ closed (2,1) form and our goal in Section \ref{SecTF} is to describe how such forms naturally appear on well known classes of Hermitian manifolds. For example, SKT manifolds, bihermitian manifolds---also known as generalized Kahler manifolds---as well as strong HKT manifolds \cite{Skt2,Rocek1,Apost1,Gualt2, Grant1} all admit $\overline{\partial}$ closed (2,1) forms and thus complex structures on their twistor spaces. 

 In Section \ref{secHSTRUC}, we construct a metric on $\mathcal{T}(TM,g)$ that is compatible with $\mathcal{J}^{(\nabla,I)}$ and compute the associated torsion $d^{c}w^{(\nabla,I)}$ in terms of the curvature and torsion of $\nabla$. ($w^{(\nabla,I)}$ is the fundamental two form.) We then focus on the case when the base manifold is strong HKT and construct a metric and three compatible complex structures on its twistor space that have equal torsions.
 
Given $E \longrightarrow (M,I)$ and a connection $\nabla$ with (1,1) curvature, in Section \ref{SecGRASS} we further study $(\mathcal{C}(E), \mathcal{J}^{(\nabla,I)})$ by holomorphically embedding it into a more familiar complex manifold. The key in finding a suitable manifold is to notice that if we $\mathbb{C}$-linearly extend $\nabla$ to a complex  connection on $E_{\mathbb{C}} := E \otimes_{\mathbb{R}} \mathbb{C}$ then $R^{\nabla}$ is (1,1) if and only if $\nabla^{0,1}$ is a $\overline{\partial}$-operator on $E_{\mathbb{C}}$. Hence given $\nabla$ on $E$ with (1,1) curvature, we have two associated complex analytic manifolds: the first is $(\mathcal{C}(E), \mathcal{J}^{(\nabla,I)})$ and the other is the holomorphic Grassmannian bundle $Gr_{n}(E_{\mathbb{C}})$ ($rankE=2n$), and in Section \ref{SecGRASS} we  holomorphically embed $\mathcal{C}(E)$ into this latter bundle. By then considering $\mathcal{C}(E)$ as a complex submanifold of $Gr_{n}(E_{\mathbb{C}})$, we derive a number of corollaries about the holomorphic structure of twistor spaces. For example, we derive conditions on two connections $\nabla$ and $\nabla'$ that are defined on $E \longrightarrow (M,I)$ with (1,1) curvature, so that the twistor spaces $(\mathcal{C}(E), \mathcal{J}^{(\nabla,I)})$ and $(\mathcal{C}(E), \mathcal{J}^{(\nabla',I)})$ are equivalent under a fiberwise biholomorphism. We then use this to prove that certain complex structures that we defined in Section \ref{SecDOP} on the twistor spaces associated to Hermitian bundles are in fact biholomorphic. As another corollary, given a holomorphic Hermitian bundle $(E,g,J)\longrightarrow (M,I)$ we construct a well defined map from the Dolbeault cohomology group $H^{0,1}(\Lambda^{2} E^{*1,0})$ to the isomorphism classes of complex structures on $\mathcal{T}(E,g)$. Other corollaries of the holomorphic embedding are given in Sections \ref{SecCD} and \ref{SecCTM}.

Having in this paper introduced, given examples and explored different properties of $(\mathcal{C}(E),\mathcal{J}^{(\nabla,I)})$, in \cite{Gindi1} we will use it to study certain Poisson structures on bihermitian manifolds.

\section{Complex Structures on Twistor Spaces}
\subsection{Preliminaries} 
\label{SecPRE}  
Let $V$ be a $2n$ dimensional real vector space and let $\mathcal{C}(V)=\{J \in EndV | \ J^{2}=-1\}$ be one of its twistor spaces. To describe some of the properties of $\mathcal{C}(V)$, consider the action of $GL(V)$ on $EndV$ via conjugation: $B \cdot A= BAB^{-1}.$  As  $\mathcal{C}(V)$ is a particular orbit of this action, it is isomorphic to 
\[GL(V)/GL(V,I),\] where $I\in \mathcal{C}(V)$ and $GL(V,I)= \{B \in GL(V)| \ [B,I]=0\} \cong GL(n,\mathbb{C})$. It then follows that the dimension of $\mathcal{C}(V)$ is $2n^{2}$ and that if we consider $\mathcal{C}(V)$ as a submanifold of $EndV$ then
\[ T_{J}\mathcal{C} = [EndV,J]=\{A \in EndV| \ \{A,J\}=0 \}.\]
With this, we may define a natural almost complex structure on $\mathcal{C}(V)$ that is well known to be integrable:
\[I_{\mathcal{C}}A = JA,  \text{ for } A \in T_{J}\mathcal{C}.\] 

If we now equip $V$ with a positive definite metric $g$ then another twistor space that we will consider is $\mathcal{T}(V,g)=\{J \in \mathcal{C}(V)| \ g(J\cdot,J\cdot)=g(\cdot,\cdot) \}.$ In this case, $\mathcal{T}$ is an orbit of the action of $O(V,g)$ on $EndV$ by conjugation, and is thus isomorphic to the Hermitian symmetric space  $O(V,g)/U(I),$ where $I \in \mathcal{T}$ and $U(I)\cong U(n)$. It then follows that the dimension of $\mathcal{T}$ is $n(n-1)$ and that if we consider $\mathcal{T}$ as a submanifold of $EndV$ then \[ T_{J}\mathcal{T}= [\mathfrak{o}(V,g),J]=\{A \in \mathfrak{o}(V,g) | \ \{A,J\}=0 \}.\] 
As  $I_{\mathcal{C}}$ naturally restricts to $T_{J}\mathcal{T}$, $\mathcal{T}$ is a complex submanifold of $\mathcal{C}$. 

\subsubsection{Twistors of Bundles}
\label{SecTB}
Let now $E \longrightarrow M$ be an even rank real vector bundle fibered over an even dimensional smooth manifold. Generalizing the previous discussion to vector bundles, we will define $\mathcal{C}(E)= \{J \in EndE | \ J^{2}=-1\}$, which is a fiber subbundle of the total space of $\pi:EndE \longrightarrow M$ with general fiber $\mathcal{C}(E_{x})$, for $x \in M$. Since the fibers of $\pi_{\mathcal{C}}:\mathcal{C}(E) \longrightarrow M$ are complex manifolds, $\mathcal{C}(E)$ naturally admits the complex vertical distribution $V\mathcal{C} \subset T\mathcal{C}(E)$, where $V_{J}\mathcal{C}= T_{J}\mathcal{C}(E_{\pi(J)}) \cong [EndE|_{\pi(J)},J]$. Using  the section $\phi \in \Gamma(\pi_{\mathcal{C}}^{*}EndE)$ defined by $\phi|_{J}=J$, we will then identify $V\mathcal{C}$ with the subbundle $[\pi_{\mathcal{C}}^{*}EndE,\phi] $ of $\pi_{\mathcal{C}}^{*}EndE$.

Letting $g$ be a positive definite fiberwise metric on $E$, we will also consider $\mathcal{T}(E,g)= \{J \in \mathcal{C}(E) | \ g(J\cdot,J\cdot)=g(\cdot,\cdot)\}$. Similar to the case of $\mathcal{C}(E)$, $\mathcal{T}(E,g)$ naturally admits the complex vertical distribution $V\mathcal{T}$, defined by $V_{J}\mathcal{T}=T_{J}\mathcal{T}(E_{\pi(J)}) \cong [\mathfrak{o}(E_{\pi(J)},g),J]$. If we denote the projection map from $\mathcal{T}(E,g)$ to $M$ by $\pi_{\mathcal{T}}$ then we will identify $V\mathcal{T}$ with the subbundle $[\pi_{\mathcal{T}}^{*}\mathfrak{o}(E,g),\phi]$ of $\pi_{\mathcal{T}}^{*}EndE$, where now $\phi \in \Gamma(\pi_{\mathcal{T}}^{*}EndE)$.

\begin{nota} As was done above and will be continued below, we will at times denote $\mathcal{C}(E)$ by $\mathcal{C}$ and $\mathcal{T}(E,g)$ by $\mathcal{T}(E)$, $\mathcal{T}(g)$ or just $\mathcal{T}$. Moreover, there are also times when we will denote $\pi_{\mathcal{C}}$ or $\pi_{\mathcal{T}}$ by just $\pi$.  
\end{nota}
\subsection{Horizontal Distributions and Splittings}
\label{SecHD}
With this background, we will now take the first steps in defining integrable complex structures on $\mathcal{C}(E)$ and $\mathcal{T}(E,g)$ in the case when $M$ is a complex manifold. Given a connection $\nabla$ on $E$ we will define the horizontal distribution $H^{\nabla}\mathcal{C}$ in $T\mathcal{C}$, so that this latter bundle splits into $V\mathcal{C} \oplus H^{\nabla}\mathcal{C}$. Similarly, in the case when $g$ is a fiberwise metric on $E$ and $\nabla$ is a metric connection, we will describe how to split $T\mathcal{T}$ into $V\mathcal{T} \oplus H^{\nabla}\mathcal{T}$. Once we have described these splittings we will define the desired complex structures on the above twistor spaces in Section \ref{SecDCS}.

To begin, let, as above, $E \longrightarrow M$ be a vector bundle, though the base manifold is not yet assumed to be a complex manifold, and let $\nabla$ be any connection. As $\mathcal{C}$ is a fiber subbundle of the total space of $\pi:EndE \longrightarrow M$, we will find it convenient to split its tangent bundle by first splitting $TEndE$. 

Although there are other ways to define this splitting the basic idea here is to use parallel translation with respect to $\nabla$.  First, if $A \in EndE$ and $\gamma:\mathbb{R} \longrightarrow M$ satisfies $\gamma(0)=\pi(A)$ then the parallel translate of $A$ along $\gamma$ will be denoted by $A(t)$. The horizontal distribution $H^{\nabla}EndE$ in $TEndE$ is then defined as follows.

\begin{defi}
Let $H^{\nabla}_{A}EndE= \{\frac{dA(t)}{dt}|_{t=0} | \text{ for all } \gamma, \gamma(0)=\pi(A)\}$. 
\end{defi}
It is straightforward to show that $H^{\nabla}EndE$ is a complement to the vertical distribution: 
\begin{lemma}
$TEndE=VEndE \oplus H^{\nabla}EndE.$ \\
\end{lemma}
\begin{rmk} The above procedure can actually be used to split the tangent bundle of any vector bundle with a connection. Another way to define such a splitting is to consider the bundle as associated to its frame bundle and then use the standard theory of connections. These two methods yield the same splittings and are essentially equivalent.
\end{rmk}

Now if $J \in \mathcal{C} \subset EndE$ and $\gamma:\mathbb{R} \longrightarrow M$ is a curve that satisfies $\gamma(0)=\pi(J)$ then it is clear that the associated parallel translate $J(t)$ lies in $\mathcal{C}$ for all relevant $t \in \mathbb{R}$. It then follows that $H^{\nabla}_{J}EndE$ lies in $T_{J}\mathcal{C}$, so that we have: 

\begin{lemma} 
\label{LemCSPLIT}
$T_{J}\mathcal{C}=V_{J}\mathcal{C} \oplus H^{\nabla}_{J}\mathcal{C}$, where $H^{\nabla}_{J}\mathcal{C}= H^{\nabla}_{J}EndE$.
\end{lemma}
Similarly, if $g$ is a fiberwise metric on $E$ and $\nabla g=0$ then the parallel translate of $J \in \mathcal{T}$ along $\gamma$ lies in $\mathcal{T}$. We thus have

\begin{lemma}
\label{LemST}
$T_{J}\mathcal{T}=V_{J}\mathcal{T} \oplus H_{J}^{\nabla}\mathcal{T}$, where $H_{J}^{\nabla}\mathcal{T}= H_{J}^{\nabla}EndE.$
\end{lemma}

With the above splittings, it will be useful for later calculations to derive a certain formula for the vertical projection operator $P^{\nabla}: TEndE \longrightarrow VEndE \cong \pi^{*}{EndE}$, which, upon suitable restriction, will also be valid for the corresponding projection operators for $T\mathcal{C}$  and $T\mathcal{T}$. The formula will depend on the tautological section $\phi$ of $\pi^{*}{EndE}$ that is defined by $\phi|_{A}=A$:

\begin{prop}
\label{PropPF}
Let $X \in T_{A}EndE,$ then \[ P^{\nabla}(X)= (\pi^{*}\nabla)_{X}\phi,\] where we are considering $P^{\nabla}$ to be a section of $T^{*}{EndE} \otimes \pi^{*}EndE$. 
\label{propP}
\end{prop}

\begin{proof}[Proof of Proposition $\ref{propP}$]
Let $\{e_{i}\}$ be a local frame for $E$ over some open set $U\subset M$ about the point $\pi(A)$, where $A \in EndE$, and let $\{e_{i} \otimes e^{j}\}$ be the corresponding frame for $EndE$. Then for $X \in T_{A}EndE$, 
\begin{align}
(\pi^{*}\nabla)_{X}\phi &=(\pi^{*}\nabla)_{X}\phi^{i}_{j}\pi^{*}(e_{i} \otimes e^{j}) \\& \label{EQ1}= d\phi^{i}_{j}(X)e_{i} \otimes e^{j}|_{\pi(A)} + A^{i}_{j} \nabla_{\pi_{*}X}e_{i} \otimes e^{j}.
\end{align}
Let us now consider the following two cases.

A) Let $X$ be an element of $V_{A}EndE$, which for the moment is not identified with $EndE|_{\pi(A)}$, so that $\pi_{*}X=0$. Also let A(t) be a curve in $EndE|_{\pi(A)}$ such that $A(0)=A$ and $\frac{dA(t)}{dt}|_{t=0}=X.$ Then by Equation \ref{EQ1}, $(\pi^{*}\nabla)_{X}\phi= \frac{dA(t)^{i}_{j}}{dt}|_{t=0}e_{i} \otimes e^{j}|_{\pi(A)}=P^{\nabla}(X) \in EndE|_{\pi(A)}$. 

B) Let $X \in H^{\nabla}_{A}EndE$ so that it equals $\frac{d}{dt}A(t)|_{t=0}$, where $A(t)$ is the parallel translate of $A$ along some curve $\gamma: \mathbb{R} \longrightarrow M$ that satisfies $\gamma(0)= \pi(A)$. As $d\phi^{i}_{j}(X) = \frac{d}{dt}A(t)^{i}_{j}|_{t=0},$ Equation \ref{EQ1} becomes $\frac{d}{dt}A(t)^{i}_{j}|_{t=0}e_{i} \otimes e^{j}|_{\pi(A)} + A^{i}_{j} \nabla_{\frac{d\gamma}{dt}|_{t=0}} e_{i} \otimes e^{j}$, which is zero since $A(t)$ is parallel.  
\end{proof}

If we consider the corresponding projection operator $P^{\nabla}: T\mathcal{C} \longrightarrow V\mathcal{C}$ then it follows from the above proposition that $P^{\nabla}(X) = (\pi_{\mathcal{C}}^{*}\nabla)_{X}\phi$, where $\phi$ is now a section of $\pi_{\mathcal{C}}^{*}EndE \longrightarrow \mathcal{C}$. Note that since $\phi^{2}=-1$, $(\pi_{\mathcal{C}}^{*}\nabla)_{X}\phi$, for $X \in T_{J}\mathcal{C}$, is indeed contained in $V_{J}\mathcal{C}=\{A \in EndE|_{\pi(J)}| \ \{A,J\}      
 = 0 \}$.  In the case when $g$ is a fiberwise metric on $E$ and $\nabla g=0$, an analogous formula holds for $P^{\nabla}: T\mathcal{T} \longrightarrow V\mathcal{T}$.  
 
 \begin{rmk}
We respectfully report that similar formulas for the projection operators for $T\mathcal{C}$ and $T\mathcal{T}$ were derived in \cite{Raw1} but with a small error.
\end{rmk}
\subsection{The Complex Structures}
\label{SecDCS}
Now let $E \longrightarrow (M,I)$ be an even rank real vector bundle that is fibered over an almost complex manifold and let $\nabla$ be a connection on $E$. We will define the following almost complex structure on the total space of $\pi: \mathcal{C}(E) \longrightarrow M$ and will explore its integrability conditions in the next section.

\begin{defi} 
\label{JI}
$\boldsymbol{\mathcal{J}^{(\nabla,I)}:}$
First use $\nabla$ to split 
\begin{equation*}
T\mathcal{C}=V\mathcal{C} \oplus H^{\nabla}\mathcal{C},
\end{equation*}
and then let
\begin{align*}
&(1) \ \mathcal{J}^{(\nabla,I)}A=JA \\
&(2) \ \mathcal{J}^{(\nabla,I)}v^{\nabla}=(Iv)^{\nabla}, 
\end{align*}       
where $A \in V_{J}\mathcal{C} \subset EndE|_{\pi(J)}$ and $v^{\nabla} \in H^{\nabla}_{J}\mathcal{C}$ is the horizontal lift of $v \in T_{\pi(J)}M.$

\end{defi}

In other words, $\mathcal{J}^{(\nabla,I)}$ on $V\mathcal{C} \oplus H^{\nabla}\mathcal{C}$ equals $\phi \oplus \pi^{*}I$, where we have identified  $V\mathcal{C}$ with $[\pi^{*}EndE,\phi]$ and $H^{\nabla}\mathcal{C}$ with $\pi^{*}TM$.  

It then follows from the definition of $\mathcal{J}^{(\nabla,I)}$ that $\pi$ is pseudoholomorphic:

\begin{prop} 
\label{PropHS}
$\pi: (\mathcal{C}, \mathcal{J}^{(\nabla,I)}) \longrightarrow (M,I)$ is a pseudoholomorphic submersion.
\end{prop}

In the case when $g$ is a fiberwise metric on $E$ and $\nabla$ is a metric connection, the claim is that $\mathcal{J}^{(\nabla,I)}$ on $\mathcal{C}$ restricts to $\mathcal{T}$, so that $\mathcal{T} \subset (\mathcal{C},\mathcal{J}^{(\nabla,I)})$ is an almost complex submanifold. The reason is that $T_{J}\mathcal{T}$ splits into $V_{J}\mathcal{T} \oplus H_{J}^{\nabla}\mathcal{T}$, where $H_{J}^{\nabla}\mathcal{T}= H_{J}^{\nabla}\mathcal{C}=H_{J}^{\nabla}EndE,$ as explained in the previous section.

\begin{rmk}
It should be noted that $\mathcal{J}^{(\nabla,I)}$ has not yet been studied in this generality in the literature. In \cite{Vais1}, Vaisman  did study $\mathcal{J}^{(\nabla,I)}$ only in the special case when $E=TM$ and $\nabla I=0$ and only on certain submanifolds of $\mathcal{C}(TM)$. However, for our applications we do not want to restrict ourselves to $E=TM$ and we especially do not want to require $\nabla I=0$. 
\end{rmk}

With $\mathcal{J}^{(\nabla,I)}$ defined, let us now compare it to the tautological almost complex structures on twistor spaces that are usually considered in the literature \cite{Ahs1,Raw1,Ber1}. If $\nabla'$ is a connection on $TM \longrightarrow M$, where here $M$ is any even dimensional smooth manifold, then based on the splitting of $T\mathcal{C}$ into $V\mathcal{C} \oplus H^{\nabla'}\mathcal{C}$, we define $\mathcal{J}^{\nabla'}_{taut}$ on $\mathcal{C}(TM)$ as follows.

\begin{defi} 
\label{DefJTAUT}
Let $\mathcal{J}^{\nabla'}_{taut}= \phi \oplus \phi$, where we have identified  $V\mathcal{C}$ with $[\pi^{*}EndTM,\phi]$ and $H^{\nabla'}\mathcal{C}$ with $\pi^{*}TM$, and where the first $\phi$ factor acts by left multiplication.
\end{defi}

To compare it to $\mathcal{J}^{(\nabla,I)}$, note that $\mathcal{J}^{\nabla'}_{taut}$ does not require $M$ to admit an almost complex structure, while the former one does. On the other hand, $\mathcal{J}^{\nabla'}_{taut}$ is only defined for the bundle $E=TM$ whereas $\mathcal{J}^{(\nabla,I)}$ is defined for any even rank real vector bundle. Also, given $(M,I)$, the projection map $(\mathcal{C}(TM), \mathcal{J}^{\nabla'}_{taut}) \longrightarrow (M,I)$ is never pseudoholomorphic, whereas $(\mathcal{C}(E), \mathcal{J}^{(\nabla,I)}) \longrightarrow (M,I)$ is always so. Lastly, $\mathcal{J}^{\nabla'}_{taut}$ is rarely integrable---except in special cases such as when $M$ is an anti-selfdual four manifold (\cite{Ahs1}), as explained in the Introduction---whereas the integrability conditions of $\mathcal{J}^{(\nabla,I)}$ are very natural to be fulfilled, as we will show below.

Having compared the above almost complex structures, let us now return to the general setup of a vector bundle $E \longrightarrow (M,I)$ that is fibered over an almost complex manifold and that is equipped with a connection $\nabla$. The goal is to determine the conditions on $I$ and the curvature of $\nabla$, $R^{\nabla}$, that are equivalent to the integrability of $\mathcal{J}^{(\nabla,I)}$ not just on $\mathcal{C}$ but on other almost complex submanifolds $\mathcal{C}'$ as well. Although these conditions can be worked out for any $\mathcal{C}'$, we will focus on the case when the corresponding projection map $\pi_{\mathcal{C}'}: \mathcal{C}' \longrightarrow M$ is a surjective submersion. If $g$ is a fiberwise metric on $E$ and $\nabla g=0$ then as an example we can take  $\mathcal{C}'= \mathcal{T}(g)$. 

The method that we will use to explore the integrability conditions of $\mathcal{J}^{(\nabla,I)}$ on $\mathcal{C}'$ is to calculate its Nijenhuis tensor on $\mathcal{C}$.

\subsubsection{Nijenhuis Tensor} 

In this section, let $\pi:\mathcal{C}(E) \longrightarrow M$ be the projection map and define $\mathcal{J}:=\mathcal{J}^{(\nabla,I)}$ and $P:=P^{\nabla}: T\mathcal{C} \longrightarrow V\mathcal{C} \subset \pi^{*}EndE$, as in Section \ref{SecHD}. 
 We will presently compute the Nijenhuis tensor, $N^{\mathcal{J}}$, of $\mathcal{J}$ that is given by
\begin{equation*}
N^{\mathcal{J}}(X,Y)=[\mathcal{J}X,\mathcal{J}Y]-\mathcal{J}[\mathcal{J}X,Y]-\mathcal{J}[X,\mathcal{J}Y]-[X,Y],
\end{equation*}
in terms of the Nijenhuis tensor of $I$ and the curvature of $\nabla$, $R^{\nabla}.$ 

\begin{prop} 
\label{PropNT}
Let $X,Y \in T_{J}\mathcal{C}$ and let $v=\pi_{*}X$ and $w=\pi_{*}Y$. Then 
\begin{align*}
&1) \ \pi_{*}N^{\mathcal{J}}(X,Y) =N^{I}(v,w)\\
&2) \ PN^{\mathcal{J}}(X,Y)=[R^{\nabla}(v,w)- R^{\nabla}(Iv,Iw), J] +J[R^{\nabla}(Iv,w)+ R^{\nabla}(v,Iw),J].
\end{align*}
\end{prop}
\begin{proof}[Proof of Proposition \ref{PropNT}, Part 1)] This easily follows from the fact that if $X \in \Gamma(T\mathcal{C})$ is $\pi \text{-related}$ to $v \in \Gamma(TM)$ then $\mathcal{J}X$ is $\pi \text{-related}$ to $Iv$. 
\end{proof}
Letting, as above, $\phi \in \Gamma(\pi^{*}EndE)$ be defined by $\phi|_{J}=J$, the proof of Part 2 of the proposition, will be based on the following lemma.
\begin{lemma} Let $X,Y \in \Gamma(T\mathcal{C})$. Then
\[P^{\nabla}([X,Y])= -[R^{\pi^{*}\nabla}(X,Y),\phi] + \pi^{*}\nabla_{X}P(Y)- \pi^{*}\nabla_{Y}P(X).\]
 
\end{lemma}
\begin{proof}
Consider 
\begin{align*}
&P^{\nabla}([X,Y])= \pi^{*}\nabla_{[X,Y]}\phi \\ &= -R^{(\pi^{*}\nabla,\pi^{*}EndE)}(X,Y)\phi +\pi^{*}\nabla_{X}\pi^{*}\nabla_{Y}\phi -\pi^{*}\nabla_{Y}\pi^{*}\nabla_{X}\phi,
\end{align*}
where $R^{(\pi^{*}\nabla,\pi^{*}EndE)}$ is the curvature of $\pi^{*}\nabla$, which is considered as a connection on $\pi^{*}EndE$. The lemma then follows from the identity:  \\$R^{(\pi^{*}\nabla,\pi^{*}EndE)}(X,Y)\phi = [R^{\pi^{*}\nabla}(X,Y),\phi]$.

\end{proof}
\begin{proof}[Proof of Proposition \ref{PropNT}, Part 2)]
Let $X,Y \in \Gamma(T\mathcal{C})$ and consider \[PN^{\mathcal{J}}(X,Y)=P([\mathcal{J}X,\mathcal{J}Y]- \mathcal{J}[\mathcal{J}X,Y]-      
 \mathcal{J}[X,\mathcal{J}Y]-[X,Y]).\] By using the previous lemma as well as the fact that $P\mathcal{J}=\phi P$, we can express $PN^{\mathcal{J}}(X,Y)$ as the sum of two sets of terms.
 The first set involves the curvature of $\pi^{*}\nabla$:
\[[R^{\pi^{*}\nabla}(X,Y)- R^{\pi^{*}\nabla}(\mathcal{J}X,\mathcal{J}Y),\phi] +\phi[R^{\pi^{*}\nabla}(\mathcal{J}X,Y)+ R^{\pi^{*}\nabla}(X,\mathcal{J}Y),\phi].\] When restricted to $J\in \mathcal{C}$ this gives the expression for $PN^{\mathcal{J}}(X,Y)$ that is contained in Part 2 of the proposition.

The second set of terms is 
\[\pi^{*}\nabla_{\mathcal{J}X}P(\mathcal{J}Y) -\phi\pi^{*}\nabla_{\mathcal{J}X}P(Y) -\phi\pi^{*}\nabla_{X}P(\mathcal{J}Y) -\pi^{*}\nabla_{X}P(Y) -(X \leftrightarrow Y). \]
Using $P\mathcal{J}=\phi P$, it easily follows that the first four terms and the last four, which are represented by $(X \leftrightarrow Y)$,  separately add to zero.    
\end{proof}
\subsubsection{Integrability Conditions}

We are now prepared to explore the integrability conditions of $\mathcal{J}^{(\nabla,I)}$ on $\mathcal{C}',$ where, as above, $\mathcal{C}'$ is any almost complex submanifold of $(\mathcal{C}(E),\mathcal{J}^{(\nabla,I)})$ such that $\pi_{\mathcal{C}'}: \mathcal{C}' \longrightarrow M$ is a surjective submersion. As is well known, $\mathcal{J}^{(\nabla,I)}$ on $\mathcal{C}'$ will be integrable if and only if $\pi_{*}N^{\mathcal{J}}(X,Y)$ and $PN^{\mathcal{J}}(X,Y)$ are both zero $\forall X,Y \in T_{J}\mathcal{C}'$ and $\forall J \in \mathcal{C}'$. By Proposition \ref{PropNT}, the first condition is equivalent to the vanishing of the Nijenhuis tensor of $I$, while the second is equivalent to  \[[R^{\nabla}(v,w)- R^{\nabla}(Iv,Iw), J] +J[R^{\nabla}(Iv,w)+ R^{\nabla}(v,Iw),J]=0 \] $\forall v,w \in T_{\pi(J)}M$ and $\forall J \in \mathcal{C}'$. To analyze this condition, we will express it in terms of $R^{0,2}$, the (0,2)-form part of the curvature $R^{\nabla}$:
\begin{lemma}
The condition \[[R^{\nabla}(v,w)- R^{\nabla}(Iv,Iw), J] +J[R^{\nabla}(Iv,w)+ R^{\nabla}(v,Iw),J]=0 \]  $\forall v,w \in T_{\pi(J)}M$ holds true if and only if  
\[[R^{0,2},J]E^{0,1}_{J}=0. \]  
\end{lemma}

We thus have:
\begin{thm}
\label{ThmICS}
$(\mathcal{C}', \mathcal{J}^{(\nabla,I)})$ is a complex manifold if and only if 
\begin{align*}
& 1) I \text{ is integrable} \\ 
& 2) [R^{0,2},J]E^{0,1}_{J}=0, \ \  \forall J \in \mathcal{C}'.
\end{align*}
\end{thm}
Note that the second condition in the above theorem is equivalent to $R^{0,2}:E^{0,1}_{J} \longrightarrow E^{0,1}_{J}, \forall J \in \mathcal{C}'.$

\subsection{(1,1) Curvature}

 Assuming henceforth that $I$ is integrable, an important case of Part 2 of the above theorem that guarantees that $(\mathcal{C}', \mathcal{J}^{(\nabla,I)})$ is a complex manifold is when $R^{(0,2)}=0$, or equivalently, when $R^{\nabla}$ is (1,1) with respect to $I$. In particular, we have:

\begin{thm}
\label{Thm1c}
Let $E \longrightarrow (M,I)$  be fibered over a complex manifold and let $\nabla$ be a connection on $E$ that has (1,1) curvature. Then $\mathcal{J}^{(\nabla,I)}$ is an integrable complex structure on $\mathcal{C}(E)$. In addition, if $g$ is a fiberwise metric on $E$ and $\nabla g=0$ then $\mathcal{T}(E,g)$ is a complex submanifold of $(\mathcal{C}(E),  \mathcal{J}^{(\nabla,I)})$.
\end{thm} 

If we $\mathbb{C}$-linearly extend $\nabla$ to a complex connection on $E_{\mathbb{C}}:=E \otimes_{\mathbb{R}}\mathbb{C}$ then the condition that $R^{\nabla}$ is (1,1) can also be expressed as $(\nabla^{0,1})^{2}=0.$ We thus have:
\begin{lemma}
\label{LemDO}
Let $\nabla$ be a connection on $E \longrightarrow (M,I)$. Then $R^{\nabla}$ is (1,1) if and only if $\nabla^{0,1}$ is a $\overline{\partial}-$operator on $E_{\mathbb{C}}$.   
\end{lemma}

In Section \ref{SecTHE} we will use the fact that $\nabla^{0,1}$ is a $\overline{\partial}-$operator to holomorphically embed $(\mathcal{C}, \mathcal{J}^{(\nabla,I)})$ into a more familiar complex manifold that is associated with the holomorphic bundle $E_{\mathbb{C}}$---the Grassmannian bundle $Gr_{n}(E_{\mathbb{C}})$. 

\begin{example}[Pseudoholomorphic Curves]
Let $E \longrightarrow (M,I)$ be an even rank real vector bundle fibered over a complex curve.  If $\nabla$ is any connection on $E$ then $R^{0,2}$ is automatically zero and hence $(\mathcal{C}(E),\mathcal{J}^{(\nabla,I)})$ is a complex manifold. Moreover if $g$ is a fiberwise metric on $E$ and $\nabla$ is a metric connection then $\mathcal{T}(E,g)$ is a complex submanifold of $(\mathcal{C}(E),\mathcal{J}^{(\nabla,I)})$. 

As an application, let $E \longrightarrow (N,J)$ be an even rank real vector bundle that is fibered over an almost complex manifold and let $\nabla$ be any connection on $E$. The goal is to show that although $(\mathcal{C}(E),\mathcal{J}^{(\nabla,J)})$ is only an almost complex manifold, it always contains many pseudoholomorphic submanifolds that are in fact complex manifolds. The idea is to use the well known existence of a plethora of pseudoholomorphic curves in $N$. Indeed, if we let $i: (S,I) \longrightarrow (N,J)$ be a pseudoholomorphic embedding of a complex curve into $N$ then the curvature of $i^{*}\nabla$ on $i^{*}E$  is (1,1) and thus $(\mathcal{C}(i^{*}E), \mathcal{J}^{(i^{*}\nabla,I)})$ is a complex manifold. As it is  straightforward to show that $i$ induces a pseudoholomorphic embedding of $\mathcal{C}(i^{*}E)$ into $\mathcal{C}(E)$, $\mathcal{C}(i^{*}E)$ is one of many examples of pseudoholomorphic submanifolds of $\mathcal{C}(E)$ that are themselves complex manifolds. 

Further connections between twistors and pseudoholomorphic curves will be explored in the near future. 
\qed 
\end{example}

In Section \ref{SecME}, we will use Theorem \ref{Thm1c} to construct complex structures on the twistor spaces of SKT, bihermitian and strong HKT manifolds.

\subsection{Other Curvature Conditions}
\label{SecOCC}
Although, by Theorem \ref{ThmICS}, the condition $R^{(0,2)}=0$ guarantees the integrability of $\mathcal{J}^{(\nabla,I)}$ on $\mathcal{C}' \subset \mathcal{C}(E)$, it is not the most general one. The present goal is to demonstrate some of these more general conditions for certain $\mathcal{C}'$.

As a first example, consider a $\mathcal{C}'$ that satisfies the following condition: given any $J \in \mathcal{C}'$, $-J$ is also in $\mathcal{C}'$. 

\begin{prop}
\label{PropIC'}
If $\mathcal{C}'$ satisfies the above condition then $(\mathcal{C}',\mathcal{J}^{(\nabla,I)})$ is a complex manifold if and only if $[R^{0,2},J]=0$ for all $J\in \mathcal{C}'$.
\end{prop}
\begin{proof}
If $(\mathcal{C}',\mathcal{J}^{(\nabla,I)})$ is a complex manifold then given $J\in \mathcal{C}'$, it follows from Theorem \ref{ThmICS}  that $[R^{0,2},J]E_{J}^{0,1}$  and $[R^{0,2},J]E_{-J}^{0,1}$ are both zero. Hence $[R^{0,2},J]=0$ for all $J\in \mathcal{C}'$. As $I$ is already assumed to be integrable, the converse also follows from Theorem \ref{ThmICS}. 
\end{proof}

In the case when $\mathcal{C}'=\mathcal{C}$, it is straightforward to show that the condition $[R^{0,2},J]=0$ for all $J \in \mathcal{C}$ is equivalent to the  endomorphism part of $R^{0,2}$ being pointwise constant. We thus have:

\begin{thm}$(\mathcal{C},\mathcal{J}^{(\nabla,I)})$ is a complex manifold if and only if $R^{(0,2)}= \lambda \otimes \mathbf{1}$, where $\lambda$ is a (0,2) form on $M$ and $\mathbf{1}$ is the identity endomorphism on $E_{\mathbb{C}}$.
\end{thm}

\begin{example}
To take a simple example, let $\nabla'$ be a connection on $E \longrightarrow (M,I)$ that has (1,1) curvature and let $\nabla=\nabla' +(w \otimes \mathbf{1})$ for some 1-form $w$. Then  $(R^{\nabla})^{0,2}=(\nabla^{0,1})^{2}$ on $E_{\mathbb{C}}$ equals $\overline{\partial}w^{0,1} \otimes \mathbf{1}$ and hence $\mathcal{J}^{(\nabla,I)}$ is a complex structure on $\mathcal{C}$. This complex structure, however, is not new since $\mathcal{J}^{(\nabla,I)}$ is actually equal to $\mathcal{J}^{(\nabla',I)}$. The reason is that although the connections $\nabla$ and  $\nabla'$ are not equal on $E$ they are in fact the same on $EndE$.

More interesting examples will be the subject of future work. 
\qed
\end{example}
  
  For another example of a $\mathcal{C'}$ of the above type, let $g$ be a fiberwise metric on $E \longrightarrow (M,I)$ and let $\nabla$ be a metric connection. As in the case for $\mathcal{C}$, it follows from Proposition $\ref{PropIC'}$ that $\mathcal{J}^{(\nabla,I)}$ is integrable on $\mathcal{C}' =\mathcal{T}(g)$ if and only if the endomorphism part of $R^{0,2}$ is pointwise constant. However, in this case $R^{0,2}$ is a (0,2) form that takes values in the skew endomorphism bundle $\mathfrak{o}(E_{\mathbb{C}},g)$, so that its trace is zero. We thus have:

 \begin{thm}
 \label{Prop1c}
 $(\mathcal{T},\mathcal{J}^{(\nabla,I)})$ is a complex manifold if and only if $R^{(0,2)}=0.$
 \end{thm}

\section{Examples}
\label{SecME}
The goal of the next few sections is to  describe various connections with (1,1) curvature on holomorphic Hermitian bundles and the resulting complex structures on the twistor spaces $\mathcal{C}$ and $\mathcal{T}$. In particular, we will demonstrate that the twistor spaces of SKT, bihermitian and strong HKT manifolds naturally admit complex structures. In Section \ref{secHSTRUC}, we will explore properties of Hermitian structures on these twistor spaces.

We will begin by considering the Chern connections of Hermitian bundles. 
\subsection{Chern Connections}
\label{SecCC}
 Let $E \longrightarrow (M,I)$ be a holomorphic vector bundle fibered over a complex manifold. Here, we will view it as a real bundle equipped with a fiberwise complex structure, $J$. If $g$ is any fiberwise metric on $E$ that is compatible with $J$ then, as is well known, the associated Chern connection $\nabla^{Ch}$ (considered as a real connection on $E$) has (1,1) curvature. We thus have
  \begin{cor}
 $(\mathcal{C}, \mathcal{J}^{(\nabla^{Ch},I)})$ is a complex manifold and $\mathcal{T}$ is a complex submanifold.
\end{cor}
\begin{example}
As a simple example, let $(M,I)$ be any complex manifold that admits a Kahler metric $g$. Then the Chern connection, $\nabla^{Ch}$, on $TM$ is the same as the Levi Civita connection, $\nabla^{Levi}$. Thus $\mathcal{J}^{(\nabla^{Levi},I)}$ is an integrable complex structure on $\mathcal{C}$ and $\mathcal{T}$.
\qed
\end{example}

 If we now $\mathbb{C}$-linearly extend $\nabla^{Ch}$ to $E_{\mathbb{C}}$ then, as a particular case of Lemma \ref{LemDO}, $\nabla^{Ch(0,1)}$ is a $\overline{\partial}$-operator for this bundle. To describe this $\bar{\partial}$-operator in more familiar terms, let us consider the holomorphic bundle $E^{1,0} \oplus E^{*1,0}$, where $E^{1,0}$ is the $+i$ eigenbundle of $J$.  The claim then is that the map
\[1 \oplus g: E_{\mathbb{C}}= E^{1,0} \oplus E^{0,1} \longrightarrow E^{1,0} \oplus E^{*1,0}\] 
is an isomorphism of holomorphic vector bundles. If we denote the Chern connection on $E^{1,0}$  by $\tilde{\nabla}^{Ch}$ then this follows from the following proposition, whose proof is straightforward.

\begin{prop} $\nabla^{Ch}= \tilde{\nabla}^{Ch} \oplus g^{-1}\tilde{\nabla}^{Ch}g$, as complex connections on $E_{\mathbb{C}}= E^{1,0} \oplus E^{0,1}$.
\end{prop} 
Thus in particular if $\{e_{i}\}$ is a local holomorphic trivialization of $E^{1,0}$ then $\{e_{i}, g^{-1}(e^{i})\}$ is a holomorphic trivialization of $E_{\mathbb{C}}$.

Now if $g'$ is another fiberwise metric on $E$ that is compatible with $J$ then in Section \ref{SecGRASS} we will address the question of whether $(\mathcal{T}(g'), \mathcal{J}^{(\nabla^{Ch'},I)})$ is biholomorphic to $(\mathcal{T}(g), \mathcal{J}^{(\nabla^{Ch},I)})$ by holomorphically embedding twistor spaces into  Grassmannian bundles.
\subsection{ $\bar{\partial}$-operators}
\label{SecDOP}
In the previous section, we found it useful to describe $\nabla^{Ch(0,1)}$ on $E_{\mathbb{C}}$ by considering the natural $\bar{\partial}$-operator $\bar{\partial}$ on $E^{1,0} \oplus E^{*1,0}$ and the isomorphism \[1 \oplus g: E_{\mathbb{C}}= E^{1,0} \oplus E^{0,1} \longrightarrow E^{1,0} \oplus E^{*1,0}.\] In this section, we will give more examples of $\bar{\partial}-$operators on $E^{1,0} \oplus E^{*1,0}$ and use this same isomorphism to transfer them to ones on $E_{\mathbb{C}}$. These in turn will give metric connections on $E$ with (1,1) curvature  that can be used to define complex structures on $\mathcal{T}$. 

To begin, let $(E,g,J)\longrightarrow (M,I)$ be, as above, a holomorphic Hermitian vector bundle and consider the following natural symmetric bilinear form $<,>$ on $E^{1,0} \oplus E^{*1,0}$: $<X+ \mu, Y+ \nu >= \frac{1}{2}(\mu(Y)+\nu(X))$. A general $\bar{\partial}$-operator that preserves this metric is of the form $\bar{\partial} + \mathcal{D}'^{0,1}$, where $\mathcal{D}'^{0,1} \in \Gamma(T^{*0,1} \otimes \mathfrak{so}(E^{1,0} \oplus E^{*1,0})).$ If we now consider the splitting of $\mathfrak{so}(E^{1,0} \oplus E^{*1,0})= EndE^{1,0} \oplus \wedge^{2}E^{*1,0} \oplus \wedge^{2}E^{1,0}$ then we may  decompose 
\[ \mathcal{D}'^{0,1}= \left(\begin{array}{cc} A &  \alpha \\ D & -A^{t}  \end{array}\right),\]
where $A, D$ and $\alpha$ are (0,1) forms with values in $ EndE^{1,0}, \wedge^{2}E^{*1,0}$ and $\wedge^{2}E^{1,0}$, respectively. 

Since $\bar{\partial} + \mathcal{D}'^{0,1}$ squares to zero, there are differential conditions on these sections. If we take, for example, the case when $\mathcal{D}'^{0,1}= D$ then these conditions are equivalent to $\overline{\partial}D=0$; a similar statement holds for the case when $\mathcal{D}'^{0,1}=\alpha.$ 

To obtain $\overline{\partial}$-operators on $E_{\mathbb{C}}$, consider, as above, the isomorphism,
\[1 \oplus g: (E_{\mathbb{C}}= E^{1,0} \oplus E^{0,1}, \frac{g}{2}) \longrightarrow (E^{1,0} \oplus E^{*1,0}, <,>).\]
$\overline{\partial} + \mathcal{D}'^{0,1}$ on $E^{1,0} \oplus E^{*1,0}$ then corresponds to $\nabla^{Ch(0,1)} + \mathcal{D}_{g}^{0,1}$ on $E_{\mathbb{C}}$, where 
\[\mathcal{D}_{g}^{0,1}= \left(\begin{array}{cc} A &  \alpha g \\ g^{-1}D & -g^{-1}A^{t}g  \end{array}\right).\] 

As we are interested in real connections on $E$, note that $\nabla^{Ch(0,1)} + \mathcal{D}_{g}^{0,1}$ is the (0,1) part of the real connection $\nabla^{Ch} + \mathcal{D}_{g}:=\nabla^{Ch} + \mathcal{D}_{g}^{0,1} + \overline{\mathcal{D}_{g}^{0,1}}$, whose curvature is (1,1).
\begin{cor}
$\mathcal{J}^{(\nabla^{Ch}+\mathcal{D}_{g}, I)}$ is a complex structure on $\mathcal{C}$ and $\mathcal{T}$.
\end{cor}

For convenience, we summarize the $\bar{\partial}$-operators and connections that we have discussed so far in the following table. 

\begin{center}
    \begin{tabular}{ | l | l | l | p{5cm} |}
    \hline
    $E$ & $E_{\mathbb{C}}$ & $E^{1,0} \oplus E^{*1,0}$ \\ \hline
     $\nabla^{Ch} + \mathcal{D}_{g}$ & $\nabla^{Ch(0,1)} + \mathcal{D}_{g}^{0,1}$ & $\bar{\partial} + \mathcal{D}'^{0,1}$    
   \\ \hline
    \end{tabular}
\end{center}

If we now take the case when $\mathcal{D}'^{0,1}=D$ then in Section \ref{SecCE} we will explore how $\mathcal{J}^{(\nabla^{Ch}+\mathcal{D}_{g}, I)}$ on $\mathcal{T}$ depends on the Dolbeault cohomology class of $D$ in $H^{0,1}(\wedge^{2}E^{*1,0})$, i.e.  if $B \in \Gamma(\wedge^{2}E^{*1,0})$ then we will determine whether $\overline{\partial} + D$ and $\overline{\partial} + D+ \overline{\partial}B$ give isomorphic complex structures on $\mathcal{T}$.  

Moreover we will also address a question that is a generalization of the one raised in the previous section: if $g'$ were another fiberwise metric on $E$ that is compatible with $J$ then given $\mathcal{D}'^{0,1} \in \Gamma(T^{*0,1} \otimes \mathfrak{so}(E^{1,0} \oplus E^{*1,0}))$, is it true that $(\mathcal{T}(g'),\mathcal{J}^{(\nabla^{Ch'}+\mathcal{D}_{g'}, I)})$ is biholomorphic to $(\mathcal{T}(g),\mathcal{J}^{(\nabla^{Ch}+\mathcal{D}_{g}, I)})$?
\subsection{Three Forms}
\label{SecTF}
An important case of the above discussion is when $E=TM$ is fibered over a Hermitian manifold $(M,g,I)$ that is equipped with a real three form $H= \overline{H^{2,1}} +H^{2,1}$ of type (1,2) + (2,1), such that $\overline{\partial}H^{2,1}=0$. In this case, we will let $\mathcal{D}'^{0,1}=H^{2,1}$, which is defined to be a section of $T^{*0,1} \otimes \mathfrak{so}(T^{1,0} \oplus T^{*1,0})$ by setting $H^{2,1}_{v}w=H^{2,1}(v,w,\cdot)$, for $v \in T^{0,1}$ and $w \in T^{1,0}$. It then follows that $\nabla^{Ch(0,1)} + g^{-1}H^{2,1}$, where here $g^{-1}H^{2,1}=g^{-1}H^{2,1}_{\frac{1}{2}(1+iI)}$, is a $\overline{\partial}$-operator on $TM_{\mathbb{C}}=T^{1,0} \oplus T^{0,1}$. As the corresponding $\mathcal{D}_{g}$ in the above table is $\frac{1}{2}I[g^{-1}H,I]$, we have  
\begin{prop}
\label{PropHCON}
 $\nabla^{Ch}+ \frac{1}{2}I[g^{-1}H,I]$ is a metric connection on $TM$ with (1,1) curvature. 
\end{prop}
 Hence $ \mathcal{J}^{(\nabla^{Ch}+ \frac{1}{2}I[g^{-1}H,I],I)}$ is a complex structure on $\mathcal{C}$ and $\mathcal{T}$. 

As we will now show, natural examples of the above three form $H$ can be found on SKT manifolds, bihermitian manifolds and strong HKT manifolds.  

\subsubsection{SKT Manifolds} 
\label{SecSKT}
A natural example of a real three form on any Hermitian manifold, $(M,g,I)$, is $H=-d^{c}w=i(\partial -\overline{\partial})w$, where $w(\cdot,\cdot)=g(I\cdot,\cdot)$. If we take its (2,1) part, $H^{2,1}$, then it is straightforward to check that it is $\overline{\partial}$ closed if and only if $dH=0$. Manifolds whose $H$ satisfy this condition are known in the literature as strong Kahler with torsion (SKT) manifolds \cite{Skt1,Skt2}. One of the associated $\overline{\partial}$-operators on $TM_{\mathbb{C}}= T^{1,0} \oplus T^{0,1}$ is $\nabla^{Ch (0,1)} -g^{-1}H^{2,1}$. As a corollary of the above discussion, this $\overline{\partial}$-operator leads to complex structures on the twistor spaces $\mathcal{C}$ and $\mathcal{T}$ that can be described as follows. First note that $\nabla^{Ch (0,1)} -g^{-1}H^{2,1}$ is the (0,1) part of the real connection $\nabla^{Ch}- \frac{1}{2}I[g^{-1}H,I]$ which can be shown to be equal to $\nabla^{-}:= \nabla^{Levi} -\frac{1}{2}g^{-1}H$, where $\nabla^{Levi}$ is the Levi Civita connection. The connection $\nabla^{-}$ is closely related to the Bismut connection, $\nabla^{+}:= \nabla^{Levi} +\frac{1}{2}g^{-1}H$ (see below for  a general definition as well as \cite{Bismut1,Guad1}). 

\begin{thm}
\label{CorCSSKT}
If $(M,g,I)$ is SKT then $(\mathcal{C}, \mathcal{J}^{(\nabla^{-},I)})$ is a complex manifold and $\mathcal{T}$ is a complex submanifold.
  \end{thm}

The Bismut connection that was mentioned above is actually defined for any almost Hermitian manifold:
\begin{defi}
\label{DefB} Let $(M,g,I)$ be an almost Hermitian manifold.
 The Bismut connection is the unique connection, $\nabla^{+}$, on $TM$ that satisfies 
\begin{align*}
& 1) \ \nabla^{+}=\nabla^{Levi} + \frac{1}{2}g^{-1}H, \text{ where $H$ is a 3-form} \\
& 2) \ \nabla^{+}I=0.
\end{align*}
\end{defi}

It can be shown that $H$ is (1,2) +(2,1) if and only if $I$ is integrable and in this case it equals $-d^{c}w$ \cite{Guad1,Gualt1}.  
\subsubsection{Bihermitian Manifolds}
\label{SecBIH}
A source of SKT manifolds is bihermitian manifolds. They were first introduced by physicists in \cite{Rocek1}, motivated by studying certain supersymmetric sigma models, and were later found to be equivalent to (twisted) generalized Kahler manifolds \cite{Gualt1,Hitchin1} (see also \cite{Apost1}). A bihermitian manifold is by definition a Riemannian manifold $(M,g)$ that is equipped with two metric compatible complex structures $J_{+}$ and $J_{-}$ that satisfy the following conditions
\[\nabla^{+}J_{+}=0 \ \ \text{ and } \ \  \nabla^{-}J_{-}=0, \]
where $\nabla^{\pm}= \nabla^{Levi} \pm \frac{1}{2}g^{-1}H$, for a closed three form $H$.

It then follows from Definition \ref{DefB} that $\nabla^{+}$ and $\nabla^{-}$ are the respective Bismut connections for $J_{+}$ and $J_{-}$. Thus an equivalent way to express the above bihermitian conditions is 
\[H=-d^{c}_{+}w_{+}=d^{c}_{-}w_{-} \ \ \text{ and } \ \ dH=0.\]
Since $dH$ is assumed to be zero, $(g,J_{+})$ and $(g,J_{-})$ are two SKT structures for $M$ and hence by Theorem \ref{CorCSSKT} the associated twistor space $\mathcal{T}$ admits the following two complex structures that depend on the three form $H$: 
\begin{thm} \label{thmCOMBIHTWO}
$\mathcal{J}^{(\nabla^{-},J_{+})}$ and $ \mathcal{J}^{(\nabla^{+},J_{-})}$ are two complex structures on $\mathcal{C}$ and $\mathcal{T}$.
\end{thm}

\subsubsection{Strong HKT Manifolds}
Another source of SKT manifolds is strong hyperkahler with torsion (strong HKT) manifolds \cite{Grant1}. Let $(M,g,I,J,K)$ be a strong HKT manifold so that $I,J$ and $K$ are metric compatible complex structures that satisfy 
\begin{itemize}
\item $\{ I,J\}=0$ and $K=IJ$ 
\item $\nabla^{+}I=0, \nabla^{+}J=0$ and $\nabla^{+}K=0,$ where $\nabla^{+}= \nabla^{L} + \frac{1}{2}g^{-1}H$, $\nabla^{L}$ is the Levi Civita connection and $H$ is a closed three form.
\end{itemize} 
Setting $\nabla^{-}= \nabla^{L} - \frac{1}{2}g^{-1}H$ and using Theorem \ref{CorCSSKT}, we have 
\begin{thm}  \label{thmHYPINT1} 
$\mathcal{J}^{(\nabla^{-},I)}, \mathcal{J}^{(\nabla^{-},J)}$ and $\mathcal{J}^{(\nabla^{-},K)}$ are three integrable complex structures on $\mathcal{T}$. 
\end{thm}

\begin{rmk}
The above complex structures on the twistor space of a strong HKT manifold are quite different from the complex structure $\mathcal{J}_{taut}$ known in the literature \cite{Grant1}. First, $\mathcal{J}_{taut}$ is integrable generally on $S:= \{aI+bJ+cK | \ a^{2}+b^{2}+c^{2}=1\} \subset \mathcal{T}$.
 Second, it is defined by using the connection $\nabla^{+}$ to split $TS= VS \oplus H^{\nabla^{+}}S$ and then setting $\mathcal{J}_{taut}= \phi \oplus \phi$. Moreover, it is integrable on $S$ without assuming $dH=0$. This is to be compared with the complex structures of Theorem  \ref{thmHYPINT1}, which are integrable on all of $\mathcal{T}$ and are defined by using the $\nabla^{-}$ connection to split $T\mathcal{T}$.  
\end{rmk}

In Section \ref{secHERSKT1}, we construct a metric compatible with $\mathcal{J}^{(\nabla^{-},I)}, \mathcal{J}^{(\nabla^{-},J)}$ and $\mathcal{J}^{(\nabla^{-},K)}$ so that the three associated Hermitian structures have equal torsions.

\section{Hermitian Structures on $\mathcal{T}$ and their Torsion} \label{secHSTRUC}
Let $(M,g,I)$ be an almost Hermitian manifold and let $\nabla$ be a metric connection on $TM$. We will first build a metric on the total space of $\pi: \mathcal{T}(TM,g) \rightarrow M$ that will be compatible with $\mathcal{J}^{(\nabla,I)}$. To do so, consider the splitting $T\mathcal{T}=V\mathcal{T} \oplus H^{\nabla}\mathcal{T}$ into vertical and horizontal distributions. As in Section \ref{SecDCS}, we will identify $H^{\nabla}\mathcal{T}$ with $\pi^{*}TM$ and $V\mathcal{T}$ with $[\pi^{*}\mathfrak{o}(TM,g), \phi]$, where $\phi \in \Gamma (\pi^{*}End(TM))$ is defined by $\phi|_{K}=K$.

\begin{defi}
Using the splitting $T\mathcal{T}=V\mathcal{T} \oplus H^{\nabla}\mathcal{T}$, define
\[g^{\nabla}= -tr \oplus \pi^{*}g,\]
where $-tr(A,B)= -tr(AB)$ for $A,B \in V_{K}\mathcal{T}$.
\end{defi}

Using Theorem \ref{Thm1c}, we have 
\begin{prop}
$(g^{\nabla}, \mathcal{J}^{(\nabla,I)})$ is an almost Hermitian structure on $\mathcal{T}$. Moreover, it is Hermitian if $I$ is integrable and $R^{\nabla}$ is (1,1) with respect to $I$.
\end{prop}

Given $(g^{\nabla}, \mathcal{J}^{(\nabla,I)})$, we then define the associated fundamental two form
\[w^{(\nabla,I)}(\cdot,\cdot)= g^{\nabla}(\mathcal{J}^{(\nabla,I)}\cdot,\cdot).\] 
The following gives an expression for the corresponding torsion three form $d^{c}w^{(\nabla,I)}(\cdot,\cdot,\cdot):= -dw^{(\nabla,I)}(\mathcal{J}^{(\nabla,I)}\cdot,\mathcal{J}^{(\nabla,I)}\cdot,\mathcal{J}^{(\nabla,I)}\cdot)$ in terms of the curvature and torsion of $\nabla$, $R^{\nabla}$ and $T^{\nabla}$, and the vertical projection operator $P^{\nabla}$ defined by the splitting of $T\mathcal{T}=V\mathcal{T} \oplus H^{\nabla}\mathcal{T}$.
(As above, $\phi \in \Gamma (\pi^{*}End(TM))$ is defined by $\phi|_{K}=K$.)
\begin{thm} \label{thmDCW}
Let $(M,g,I)$ be an almost Hermitian manifold and let $\nabla$ be a metric connection on $TM$. Given $X_{i}\in T_{K}\mathcal{T}$ with $\pi_{*}X_{i}=v_{i}$ and  $i\in \{1,2,3\}$, 
\begin{align*}
d^{c}w^{(\nabla,I)}(X_{1},X_{2},X_{3})= \ &   
                                                   tr([R^{\nabla}(Iv_{1},Iv_{2}), \phi] P^{\nabla}(X_{3})) \\ &  + g((\nabla_{Iv_{1}} I)v_{2},v_{3}) - g(T^{\nabla}(Iv_{1},Iv_{2}),v_{3}) \\ &  + cyclic(1,2,3).
\end{align*}
\end{thm}

The proof will be based on the following proposition whose proof is straightforward. 

\begin{prop}\label{propDCWN}
Let $(N,g,J)$ be an almost Hermitian manifold and let $\nabla$ be a metric connection on $TN$. Set $w(\cdot,\cdot)=g(J\cdot,\cdot)$ and $d^{c}w(\cdot,\cdot,\cdot)=-dw(J\cdot,J\cdot,J\cdot)$. Given $X_{i}\in T_{x}N$, for $i\in \{1,2,3\}$, 
\[d^{c}w(X_{1},X_{2},X_{3})=g((\nabla_{JX_{1}} J)X_{2},X_{3}) - g(T^{\nabla}(JX_{1},JX_{2}),X_{3}) +cyclic(1,2,3).\]
\end{prop}
To apply this proposition to the setting of Theorem \ref{thmDCW}, we will define the following connection on $T\mathcal{T}$ that will be compatible with $g^{\nabla}$. First note that since the connection $\pi^{*}\nabla' :=\pi^{*}\nabla + \frac{1}{2}(\pi^{*}\nabla \phi)\phi$ on $\pi^{*}TM$ satisfies $\pi^{*}\nabla'\phi =0$, it extends to a connection on $V\mathcal{T}$. Using the splitting $T\mathcal{T}=V\mathcal{T} \oplus H^{\nabla}\mathcal{T}$, we then define \[D= \pi^{*}\nabla + \frac{1}{2}(\pi^{*}\nabla \phi)\phi \oplus \pi^{*}\nabla. \] 

\begin{prop}\label{propDCONN}
$D$ is a connection on $T\mathcal{T}$ that is compatible with $g^{\nabla}$ and satisfies $D\mathcal{J}^{(\nabla,I)}=  0 \oplus \pi^{*}\nabla \pi^{*}I$.
\end{prop}

We will now express $T^{D}$ in terms of $R^{\nabla}$ and $T^{\nabla}$ and will then prove Theorem \ref{thmDCW}.
\begin{thm} \label{thmDTOR} For $X,Y \in T_{K}\mathcal{T},$
\begin{align*}
&1) \ P^{\nabla}T^{D}(X,Y)= [\pi^{*}R^{\nabla}(X,Y),\phi]\\
&2) \ \pi_{*}T^{D}(X,Y)= \pi^{*}T^{\nabla}(X,Y).
\end{align*}
\end{thm}

\begin{proof}
To prove Part 1), consider 
\begin{align*}P^{\nabla}T^{D}(X,Y)& = P^{\nabla}(D_{X}Y-D_{Y}X - [X,Y])\\ &= D_{X}P^{\nabla}(Y)-D_{Y}P^{\nabla}(X) - P^{\nabla}([X,Y]).
\end{align*}
Using that $P^{\nabla}(X)= \pi^{*}\nabla_{X} \phi$ (see Proposition \ref{PropPF}) and the definition of $D$, this becomes 
\begin{align*}
&\pi^{*}\nabla_{X}\pi^{*}\nabla_{Y} \phi - \pi^{*}\nabla_{Y}\pi^{*}\nabla_{X} \phi - \pi^{*}\nabla_{[X,Y]}\phi
\\ &+\frac{1}{2}[(\pi^{*}\nabla_{X}\phi)\phi, \pi^{*}\nabla_{Y}\phi]- \frac{1}{2}[(\pi^{*}\nabla_{Y}\phi)\phi, \pi^{*}\nabla_{X}\phi].
\end{align*}
Since the last two terms add to zero, this equals 
\begin{align*}
R^{(\pi^{*}\nabla,End\pi^{*}TM)}(X,Y)\phi = [R^{(\pi^{*}\nabla,\pi^{*}TM)}(X,Y),\phi] = [\pi^{*}R^{(\nabla,TM)}(X,Y),\phi],
\end{align*}
where $R^{(\nabla',E)}$ is the curvature associated with a connection $\nabla'$ on a vector bundle $E$.

The proof of Part 2) is straightforward.
\end{proof}

We will now prove Theorem \ref{thmDCW}.

\begin{proof}[Proof of Theorem \ref{thmDCW}]
Set $\mathcal{J}= \mathcal{J}^{(\nabla,I)}$ and let $X_{i}\in T_{K}\mathcal{T}$ with $\pi_{*}X_{i}=v_{i}$ and  $i\in \{1,2,3\}$. By Propositions \ref{propDCWN} and \ref{propDCONN}, 
\begin{align*}
d^{c}w^{(\nabla,I)}(X_{1},X_{2},X_{3})&= g^{\nabla}((D\mathcal{J})_{\mathcal{J}X_{1}}X_{2},X_{3})\\&- g^{\nabla}(T^{D}(\mathcal{J}X_{1},\mathcal{J}X_{2}),X_{3}) + cyclic (1,2,3).
\end{align*}

Using Proposition \ref{propDCONN} and Theorem \ref{thmDTOR}, this becomes
\begin{align*}
&tr([\pi^{*}R^{\nabla}(\mathcal{J}X_{1},\mathcal{J}X_{2}),\phi]P^{\nabla}(X_{3})) + g((\pi^{*}\nabla_{\mathcal{J}X_{1}}\pi^{*}I)v_{2},v_{3})\\ & -g(\pi^{*}T^{\nabla}(\mathcal{J}X_{1},\mathcal{J}X_{2}),v_{3}) +cyclic(1,2,3).
\end{align*}
This equals 
\begin{align*}
&tr([R^{\nabla}(Iv_{1},Iv_{2}), \phi]P^{\nabla}(X_{3}))+ g((\nabla I)_{Iv_{1}}v_{2},v_{3})\\& -g(T^{\nabla}(Iv_{1},Iv_{2}),v_{3}) + cyclic(1,2,3).
\end{align*}
\end{proof}

\subsection{Hermitian Pairs with Equal Torsion}
As a first application of Theorem \ref{thmDCW}, we will construct a metric and two compatible complex structures on the twistor space that have equal torsions. To do so, let $(M,g,I)$ be a Hermitian manifold and let $\nabla$ be a metric connection on $TM$ with $(1,1)$ curvature. Along with $\mathcal{J}^{(\nabla,I)}$, the integrable complex structure $-\mathcal{J}^{(\nabla,-I)}$ is compatible with $g^{\nabla}$. Letting  $w^{(\nabla,1)}$ and $w^{(\nabla,2)}$ be the respective fundamental two forms, we have 

\begin{thm}Given $X_{i}\in T_{K}\mathcal{T}$ with $\pi_{*}X_{i}=v_{i}$ and  $i\in \{1,2,3\}$, \\
\hspace*{1cm}1) $d^{c}w^{(\nabla,1)}(X_{1},X_{2},X_{3}) = d^{c}w^{(\nabla,2)}(X_{1},X_{2},X_{3})$ 
\begin{align*} 
 2)\ d^{c}w^{(\nabla,1)}(X_{1},X_{2},X_{3}) &=tr([R^{\nabla}(v_{1},v_{2}), \phi] P^{\nabla}(X_{3})) \\ &  + g((\nabla_{Iv_{1}} I)v_{2},v_{3}) - g(T^{\nabla}(Iv_{1},Iv_{2}),v_{3}) \\ &  + cyclic(1,2,3).
\end{align*}
\end{thm}

\begin{rmk} \label{rmkNOTDC}
In this notation, the ``$d^{c}$" in $d^{c}w^{(\nabla,1)}$ is with respect to $\mathcal{J}^{(\nabla,I)}$ and that in $d^{c}w^{(\nabla,2)}$ is with respect to $-\mathcal{J}^{(\nabla,-I)}$.
\end{rmk}

\subsection{Strong HKT Manifolds} \label{secHERSKT1}
Let $(M,g,I,J,K)$ be a strong hyperkahler with torsion (strong HKT) manifold so that $I,J$ and $K$ are metric compatible complex structures that satisfy 
\begin{itemize}
\item $\{ I,J\}=0$ and $K=IJ$ 
\item $\nabla^{+}I=0, \nabla^{+}J=0$ and $\nabla^{+}K=0,$ where $\nabla^{+}= \nabla^{L} + \frac{1}{2}g^{-1}H$, $\nabla^{L}$ is the Levi Civita connection and $H$ is a closed three form.
\end{itemize} 
Setting $\nabla^{-}=\nabla^{L} - \frac{1}{2}g^{-1}H$ and using Theorem \ref{thmHYPINT1}, we have

\begin{prop}  \label{propHYPINT1234} 
$\mathcal{J}^{(\nabla^{-},I)}, \mathcal{J}^{(\nabla^{-},J)}$ and $\mathcal{J}^{(\nabla^{-},K)}$ are three integrable complex structures on $\mathcal{T}$ that are compatible with the metric $g^{\nabla^{-}}$. 
\end{prop}
The associated torsions are equal (see Remark \ref{rmkNOTDC} on notation):

\begin{thm} \label{thmHYP1} Given $X_{i}\in T_{K'}\mathcal{T}$ with $\pi_{*}X_{i}=v_{i}$ and  $i\in \{1,2,3\}$, \\
1) \ $d^{c}w^{(\nabla^{-},I)}(X_{1},X_{2},X_{3}) =  
d^{c}w^{(\nabla^{-},J)}(X_{1},X_{2},X_{3})  =
d^{c}w^{(\nabla^{-},K)}(X_{1},X_{2},X_{3}) $
\begin{align*} 
  \hspace*{-.7cm} 2) \ d^{c}w^{(\nabla^{-},I)}(X_{1},X_{2},X_{3}) &=tr([R^{\nabla^{-}}(v_{1},v_{2}), \phi] P^{\nabla^{-}}(X_{3})) + cyclic(1,2,3) \\ &  -H(v_{1},v_{2},v_{3}).
\end{align*}
\end{thm}

The proof is based on Theorem \ref{thmDCW} and the following lemma.
\begin{lemma}
Let $(M,g,I)$ be a Hermitian manifold and let $\nabla^{-}=\nabla^{L}-\frac{1}{2}g^{-1}H$, where $H=-d^{c}w$ and $w(\cdot,\cdot)=g(I\cdot,\cdot)$. For $v_{1}, v_{2}, v_{3} \in T_{x}M$,  \[-H(v_{1},v_{2},v_{3})= g(  (\nabla^{-}_{Iv_{1}}I)v_{2},v_{3}) -g( T^{\nabla^{-}}(Iv_{1},Iv_{2}),v_{3}) +cyclic(1,2,3).\]
\end{lemma}
\begin{proof}
Using $\nabla^{-}I= -[g^{-1}H,I]$ and $T^{\nabla^{-}}(v_{1},v_{2})=-g^{-1}H_{v_{1}}v_{2}$, 
\begin{align*}
&g((\nabla^{-}_{Iv_{1}}I)v_{2},v_{3}) -g( T^{\nabla^{-}}(Iv_{1},Iv_{2}),v_{3}) +cyclic(1,2,3)\\&=
-g([g^{-1}H_{Iv_{1}},I]v_{2},v_{3}) +g( g^{-1}H_{Iv_{1}}Iv_{2},v_{3}) +cyclic(1,2,3).
\end{align*}
Since $H$ is $(1,2)+(2,1)$ with respect to $I$, $[g^{-1}H_{Iv},I]= I[g^{-1}H_{v},I]$, so that the above becomes
\begin{align*}
& -g(I[g^{-1}H_{v_{1}},I]v_{2},v_{3}) +g( g^{-1}H_{Iv_{1}}Iv_{2},v_{3}) +cyclic(1,2,3)
\\&= g(g^{-1}H_{v_{1}}Iv_{2},Iv_{3}) -g(g^{-1}H_{v_{1}}v_{2},v_{3}) + g(g^{-1}H_{Iv_{1}}Iv_{2},v_{3}) + cyclic(1,2,3)
\\&=H(v_{1},Iv_{2},Iv_{3}) -H(v_{1},v_{2},v_{3})+H(Iv_{1},Iv_{2},v_{3}) +cyclic(1,2,3).
\end{align*}
Using that $H$ is $(1,2)+(2,1)$, this equals 
\begin{align*}
& -H(Iv_{1},v_{2},Iv_{3})-H(Iv_{2},v_{3},Iv_{1}) -H(Iv_{3},v_{1},Iv_{2})
\\ &= -H(Iv_{1},v_{2},Iv_{3})-H(Iv_{1},Iv_{2},v_{3})-H(v_{1},Iv_{2},Iv_{3})
\\&=-H(v_{1},v_{2},v_{3}).
\end{align*} 
\end{proof}

\subsubsection{Hyperkahler Manifolds}
If we restrict to the case when $(M,g,I,J,K)$ is hyperkahler, so that $H=0$, then the torsions associated with the metric $g^{\nabla^{L}}$ and the complex structures $\mathcal{J}^{(\nabla^{L},I)}, \mathcal{J}^{(\nabla^{L},J)}$ and $\mathcal{J}^{(\nabla^{L},K)}$ are equal and generally nonzero:

\begin{thm}  \label{thmHYP2}
Let $(M,g,I,J,K)$ be a hyperkahler manifold. Given $X_{i}\in T_{K'}\mathcal{T}$ with $\pi_{*}X_{i}=v_{i}$ and  $i\in \{1,2,3\}$, \\
1) \ $d^{c}w^{(\nabla^{L},I)}(X_{1},X_{2},X_{3}) =  
d^{c}w^{(\nabla^{L},J)}(X_{1},X_{2},X_{3})  =
d^{c}w^{(\nabla^{L},K)}(X_{1},X_{2},X_{3})$          
\begin{align*}                      
 \hspace*{-.5cm} 2)\ d^{c}w^{(\nabla^{L},I)}(X_{1},X_{2},X_{3})   =tr([R^{\nabla^{L}}(v_{1},v_{2}), \phi] P^{\nabla^{L}}(X_{3})) + cyclic(1,2,3). 
\end{align*}
\end{thm}

%

\subsection{Bihermitian Manifolds}

Let $(M,g,J_{+},J_{-})$ be a bihermitian manifold, so that \[\nabla^{+}J_{+}=0 \text{ and } \nabla^{-}J_{-}=0,\] where $\nabla^{\pm}= \nabla^{L} \pm \frac{1}{2}g^{-1}H$, $\nabla^{L}$ is the Levi Civita connection and $H$ is a closed three form. Using Theorem \ref{thmCOMBIHTWO}, we have the following two Hermitian structures on $\mathcal{T}:$ $(g^{\nabla^{+}}, \mathcal{J}^{(\nabla^{+},J_{-})})$ and $(g^{\nabla^{-}}, \mathcal{J}^{(\nabla^{-},J_{+})})$. 

The first step will be to compare the two metrics $g^{\nabla^{+}}$ and $g^{\nabla^{-}}$ in the following proposition, whose proof is straightforward.

\begin{prop} 
\mbox{} 

\hspace*{.001cm} 1) For $X,Y \in  H_{K}^{\nabla^{L}}\mathcal{T}$ with $\pi_{*}X=v$ and $\pi_{*}Y=w$, 
\begin{align*}
 g^{\nabla^{+}}(X,Y)=& g^{\nabla^{-}}(X,Y) \\
                   				=& g(v,w) -\frac{tr}{4}([g^{-1}H_{v},\phi] [g^{-1}H_{w},\phi]).
\end{align*}

\hspace*{.001cm} 2) For $A,B \in  V_{K}\mathcal{T}$, 
\begin{align*}
g^{\nabla^{+}}(A,B)=g^{\nabla^{-}}(A,B).
\end{align*}

\hspace*{.001cm}
3) For $A \in  V_{K}\mathcal{T}$ and $Y \in  H_{K}^{\nabla^{L}}\mathcal{T}$ with $\pi_{*}Y=w$, 
\begin{align*}
 g^{\nabla^{+}}(A,Y)=& -g^{\nabla^{-}}(A,Y) \\
                   				=& -\frac{tr}{2}(A[g^{-1}H_{w},\phi]).
\end{align*}
\end{prop}

Consequently, the above constructions do not yield a metric on $\mathcal{T}$ that is compatible  with both complex structures $\mathcal{J}^{(\nabla^{+},J_{-})}$ and $\mathcal{J}^{(\nabla^{-},J_{+})}$. In fact, one can show that such metrics do not exist on the twistor space of a general bihermitian manifold. (Though perhaps they exist on certain submanifolds of $\mathcal{T}$.) In the case when $(J_{+}-J_{-})$ is invertible, I have used different constructions to build such metrics on all of $\mathcal{T}$. The key is to use splittings of $T\mathcal{T}$ that are not induced from connections on $TM$. I will present these results in a separate paper. 


\section{Twistors and Grassmannians}
\label{SecGRASS}
In Sections \ref{SecCC} and \ref{SecDOP}, we raised several questions about the complex manifold structure of $(\mathcal{C}(E),\mathcal{J}^{(\nabla,I)})$, where $R^{\nabla}$ is $(1,1)$. In this section we will address these questions by holomorphically embedding $\mathcal{C}$ into a more familiar complex manifold---a certain Grassmannian bundle.   Indeed, as we noted previously, the condition that $R^{\nabla}$ is (1,1) is equivalent to $\nabla^{0,1}$ being a $\overline{\partial}$-operator on $E_{\mathbb{C}}$, and if we let $rankE=2n$ then the Grassmannian bundle that we will take will be the holomorphic bundle $Gr_{n}(E_{\mathbb{C}})$.

To define the embedding, we will first show how to holomorphically embed the fibers of $\mathcal{C}$ into those of $Gr_{n}(E_{\mathbb{C}})$. 

\subsection{Embedding the Fibers}
Let $V$ be a $2n$ dimensional real vector space and let $Gr_{n}(V_{\mathbb{C}})$ be the Grassmannians of complex $n$ planes. The map that we will consider is 
\begin{align*}
\psi: \mathcal{C}(V) &\longrightarrow Gr_{n}(V_{\mathbb{C}})\\
 J &\longrightarrow V^{0,1}_{J};
\end{align*}
it has the following properties:

\begin{prop}
\label{PropVECEM}
\begin{align*}
& \text{ 1) The map }  \psi: \mathcal{C}(V) \longrightarrow Gr_{n}(V_{\mathbb{C}}) \text{ is a holomorphic embedding. } \\
& \text{ 2) The image of } \psi \text{ is } \{ P \in Gr_{n}(V_{\mathbb{C}}) | P \oplus \overline{P}=V_{\mathbb{C}}\}, \text{which is an open}  \\ & \ \ \ \  \text{ submanifold of the Grassmannians. }
\end{align*}

\end{prop}

\begin{proof}
Consider $\psi_{*}: T_{J} \mathcal{C}(V) \longrightarrow T_{V^{0,1}_{J}}Gr_{n}(V_{\mathbb{C}})$ and choose the holomorphic chart
\begin{align*} 
End(V^{0,1}_{J},&V^{1,0}_{J}) \longrightarrow Gr_{n}(V_{\mathbb{C}}) \\
B & \longrightarrow Graph(B),
\end{align*}
where $Graph(B)=\{ v^{0,1} + Bv^{0,1}| v^{0,1} \in V^{0,1}_{J}\}.$
If we let $A$ be a general element in $T_{J} \mathcal{C}(V)\cong \{D \in EndV | \{D,J\}=0\}$ then we need to show that $\psi_{*}(JA)= \mathcal{I}\psi_{*}(A)$, where $\mathcal{I}$ is the complex structure on the Grassmannians. 

First consider, 
\begin{align*}
\psi_{*}(JA) &= \frac{d}{dt}|_{t=0}\psi(exp(\frac{-tA}{2})Jexp(\frac{tA}{2}))\\
                         &= \frac{d}{dt}|_{t=0}exp(\frac{-tA}{2})(V^{0,1}_{J}). 
  \end{align*}                       
    Using the above chart, $\psi_{*}(JA)$ then corresponds to $-\frac{A}{2}$, as an element of $End(V^{0,1}_{J},V^{1,0}_{J})$.
 
  Similarly we have $\psi_{*}(A)= \frac{d}{dt}|_{t=0}exp(-\frac{tAJ}{2})(V^{0,1}_{J})$, so that            
under the above chart, $\mathcal{I}\psi_{*}(A)$ corresponds to $-\frac{iAJ}{2}$, which as an element of $End(V^{0,1}_{J},V^{1,0}_{J})$ equals $-\frac{A}{2}$.

The proof of the other parts of the proposition is straightforward.
\end{proof}

If we now choose a positive definite metric, $g$, on $V$ then by restriction, the above map, $\psi$,  gives a holomorphic embedding of $\mathcal{T}(V)$ into $Gr_{n}(V_{\mathbb{C}})$. Since the metric is positive definite, the image of this map is precisely $MI(V_{\mathbb{C}})= \{P \in Gr_{n}(V_{\mathbb{C}}) | g(v,w)=0, \forall v,w \in P\} $, the space of maximal isotropics of $V_{\mathbb{C}}$ defined by using the $\mathbb{C}$-bilinearly extended metric. For convenience we state this as a proposition.

\begin{prop}
 \begin{align*}
 \mathcal{T}(V) &\longrightarrow Gr_{n}(V_{\mathbb{C}}) \\
     J & \longrightarrow V^{0,1}_{J}    
\end{align*}
is a holomorphic embedding with image $MI(V_{\mathbb{C}})$.
\end{prop}
\subsection{The Holomorphic Embedding}
\label{SecTHE}
Let us now consider a rank $2n$ real vector bundle $E \longrightarrow (M,I)$ that is fibered over a complex manifold. As discussed above, a connection $\nabla$ on $E$ with (1,1) curvature gives rise to two complex analytic manifolds: the twistor space $(\mathcal{C}, \mathcal{J}^{(\nabla,I)})$ and the  holomorphic fiber bundle $\pi_{Gr}:Gr_{n}(E_{\mathbb{C}}) \longrightarrow M$.  To holomorphically embed $\mathcal{C}$ into $Gr_{n}(E_{\mathbb{C}})$, we will  generalize the map $\psi$ that was defined in the previous section: 

\begin{thm}
\label{ThmGEM}
The map
\begin{align*}
\psi: (\mathcal{C}, \mathcal{J}^{(\nabla,I)})& \longrightarrow Gr_{n}(E_{\mathbb{C}}) \\
& J \longrightarrow E^{0,1}_{J}
\end{align*}
is a holomorphic embedding. 
\end{thm}

In the case when $E$ is equipped with a fiberwise metric $g$ and $\nabla$ is a metric connection, we will define $MI(E_{\mathbb{C}})$ to be the space of maximal isotropics in $Gr_{n}(E_{\mathbb{C}})$; we then have: 

 \begin{prop}
 \label{PropGEM}
 \begin{align*}
 (\mathcal{T}, \mathcal{J}^{(\nabla,I)}) &\longrightarrow Gr_{n}(E_{\mathbb{C}}) \\
J & \longrightarrow E^{0,1}_{J}    
\end{align*}
is a holomorphic embedding with image $MI(E_{\mathbb{C}})$.
\end{prop}

To prove Theorem \ref{ThmGEM}, we will need to describe the complex structure on the Grassmannians  similarly to how we defined $\mathcal{J}^{(\nabla,I)}$ on $\mathcal{C}$. The first step will be to define the horizontal distribution $H^{\nabla}Gr_{n}$ on $Gr_{n}(E_{\mathbb{C}})$. But before giving the definition, let us first recall that if $P \in Gr_{n}(E_{\mathbb{C}})$ and $\gamma: \mathbb{R} \longrightarrow M$ satisfies $\gamma(0)= \pi_{Gr}(P)$ then we can use $\nabla$, considered as a complex connection on $E_{\mathbb{C}}$, to parallel transport $P$ along $\gamma$ as follows. If we set $P=<e_{1},...,e_{n}>_{\mathbb{C}}$, so that $\{e_{i}\}$ is a basis for $P$,  then define $P(t)=<e_{1}(t),...,e_{n}(t)>_{\mathbb{C}}$, where $\gamma^{*}\nabla e_{i}(t)=0$ and $e_{i}(0)=e_{i}$. Since $\nabla$ is a complex connection on $E_{\mathbb{C}}$, $P(t)$ does not depend on the basis $\{e_{i}\}$ for $P$ that was chosen. 

With this, let us define the desired horizontal distribution on $Gr_{n}(E_{\mathbb{C}})$.

\begin{defi}
Let $H^{\nabla}_{P}Gr_{n}= \{ \frac{dP(t)}{dt}|_{t=0}| P(t) \text{ is the parallel translate of } $P$ \\ \text{ along } \gamma, \gamma(0)=\pi_{Gr}(P) \}.$ 
\end{defi}

Along with $H^{\nabla}Gr_{n}$, there is also the natural vertical distribution $VGr_{n}$; as it is defined by the fibers of $Gr_{n}(E_{\mathbb{C}})$, it is a complex vector bundle and satisfies  ${\pi_{Gr}} _{*}(V_{P}Gr_{n})=0$, for all $P \in Gr_{n}(E_{\mathbb{C}})$.   It is straightforward to prove that these two distributions are complements to each other: 
\begin{lemma}
$T_{P}Gr_{n}= V_{P}Gr_{n} \oplus H^{\nabla}_{P}Gr_{n}.$
 \end{lemma}

We may now use the above lemma to define an almost complex structure on $Gr_{n}(E_{\mathbb{C}})$, which we will show in Proposition \ref{PropCSG} to be the complex structure that is induced by $\nabla^{0,1}$ and which we will use to prove Theorem \ref{ThmGEM}. As the definition of this almost complex structure is similar to that of $\mathcal{J}^{(\nabla,I)}$ on $\mathcal{C}$, we will denote it by the same symbol:
\begin{defi}
Let $\mathcal{J}^{(\nabla,I)}$ on $Gr_{n}(E_{\mathbb{C}})$ be defined as follows. First split 
\[ TGr_{n}= VGr_{n} \oplus H^{\nabla}Gr_{n} \]
and then let  
\[\mathcal{J}^{(\nabla,I)}= \mathcal{J}^{V} \oplus \pi^{*}_{Gr}I,\]
where $\mathcal{J}^{V}$ is the standard fiberwise complex structure on $VGr_{n}$ and where we have used the natural identification of $H^{\nabla}Gr_{n}$ with  $\pi^{*}_{Gr}TM$.
\end{defi}

If we consider the complex manifold structure of $Gr_{n}(E_{\mathbb{C}})$ that is induced by the $\overline{\partial}$-operator $\nabla^{0,1}$ on $E_{\mathbb{C}}$, we then have:

\begin{prop}
\label{PropCSG}
The complex structure on $Gr_{n}(E_{\mathbb{C}})$ is $\mathcal{J}^{(\nabla,I)}$.
\end{prop}
 
 We will prove the above proposition for a more general setup in the next section; here we will use it to prove Theorem \ref{ThmGEM} by showing that the map $\psi: (\mathcal{C}, \mathcal{J}^{(\nabla,I)}) \longrightarrow (Gr_{n}(E_{\mathbb{C}}),\mathcal{J}^{(\nabla,I)})$, which is given by $\psi(J)=E^{0,1}_{J}$, is holomorphic. Recalling the splitting of $T\mathcal{C}= V\mathcal{C} \oplus H^{\nabla}\mathcal{C}$, as given in Lemma \ref{LemCSPLIT}, let us first consider the following: 

\begin{lemma}
\label{LemPH}
The map $\psi_{*}$ preserves horizontals: $\psi_{*}:H^{\nabla}_{J}\mathcal{C} \longrightarrow H^{\nabla}_{E^{0,1}_{J}}Gr_{n}$. In fact, $\psi_{*}(v^{\nabla})=v^{(\nabla,Gr)}$, where $v^{\nabla}$ and $v^{(\nabla,Gr)}$ are the appropriate horizontal lifts of $v \in T_{x}M$.  
\end{lemma}

\begin{proof}
Let $\gamma(t)$ be a curve in $M$ such that $\gamma(0)=x$ and $\gamma'(0)=v$. Also let $J(t)$ be the parallel translate of $J \in \mathcal{C}(E_{x})$ along $\gamma$ (by using $\nabla$), so that
 \[ \psi_{*}(v^{\nabla})= \frac{d}{dt}|_{t=0}\psi(J(t)).\]
 
 The claim then is that $\psi(J(t))$, which is by definition $E^{0,1}_{J(t)}$, equals $E^{0,1}_{J}(t)$, the parallel translate of $E^{0,1}_{J}$ along $\gamma$. To show this just note that if $e(t)$ is the parallel translate of $e \in E^{0,1}_{J}$ then $J(t)e(t)$ is also parallel and since $Je=-ie$, it follows that $J(t)e(t)=-ie(t)$ for all relevant $t \in \mathbb{R}$. Hence $\frac{d}{dt}|_{t=0}\psi(J(t))=\frac{d}{dt}|_{t=0}E^{0,1}_{J}(t)=v^{(\nabla,Gr)}.$
\end{proof}

Assuming Proposition \ref{PropCSG}, we can now prove that $\psi$ is holomorphic:

\begin{proof}[Proof of Theorem \ref{ThmGEM}]
Consider $\psi_{*}: T_{J} \mathcal{C} \longrightarrow T_{E^{0,1}_{J}}Gr_{n}$. By Proposition \ref{PropCSG}, we need to show that $\psi_{*}\mathcal{J}^{(\nabla,I)}= \mathcal{J}^{(\nabla,I)}\psi_{*}$.

A) If $A \in V_{J}\mathcal{C}$, the vertical tangent space to $J$, then it follows from Proposition \ref{PropVECEM} that \[ \psi_{*}(JA)=\mathcal{J}^{(\nabla,I)}\psi_{*}(A),
\]
so that $\psi_{*}$ is holomorphic in the vertical directions. 

B) As for the horizontal directions, let $v^{\nabla} \in H^{\nabla}_{J}\mathcal{C}$ be the horizontal lift of $v \in T_{x}M$. Then $\psi_{*}(\mathcal{J}^{(\nabla,I)} v^{\nabla})= \psi_{*}((Iv)^{\nabla})$, which by Lemma \ref{LemPH} equals $(Iv)^{(\nabla,Gr)}$. This in turn equals $\mathcal{J}^{(\nabla,I)}v^{(\nabla,Gr)}=\mathcal{J}^{(\nabla,I)}\psi_{*}(v^{\nabla})$.
\end{proof}
\subsection{Proof of Proposition \ref{PropCSG}} 
In this section, we will prove a slightly more general version of Proposition \ref{PropCSG}; this will then complete the proof of Theorem \ref{ThmGEM}. To begin, we will find it useful to describe the complex structures on holomorphic vector bundles:

Let $\pi_{F}: F \longrightarrow (M,I)$ be a complex vector bundle that is equipped with a $\overline{\partial}$-operator, $\overline{\partial}$, and let $\nabla$ be a complex connection on $F$ such that $\nabla^{0,1}=\overline{\partial}$. Below we will let $\mathcal{J}^{(\nabla,I)}$ be the almost complex structure on either $F$ or $Gr_{k}(F)$  that is defined in a by now familiar way: use $\nabla$ to split the appropriate tangent bundle into vertical and horizontal distributions, and define $\mathcal{J}^{(\nabla,I)}$ to be the direct sum of the given fiberwise complex structure on the verticals and the lift of $I$ on the horizontals. 

\begin{prop}Let $\nabla$ be a complex connection on $F$ such that $\nabla^{0,1}= \overline{\partial}$. Then the associated complex structure on $F$ is $\mathcal{J}^{(\nabla,I)}$.
\end{prop}
\begin{proof}
Let $\{f_{i}\}$ $(1 \leq i \leq rankF)$ be a holomorphic frame for $F$ over $U \subset M$ and let $W$ be a complex vector space with basis $\{w_{i}\}$.
To prove the proposition, we need to show that the map
\begin{align*}  
\sigma: (F|_{U},\mathcal{J}^{(\nabla,I)}&) \longrightarrow U \times W\\
a_{i}f_{i}|_{x} &\longrightarrow (x, a_{i}w_{i})
\end{align*}
is pseudoholomorphic. 
For this, consider $\sigma_{*}:T_{f}F \longrightarrow T_{\sigma(f)}(U \times W)$, where $\pi_{F}(f)=x$.

1) Since $\sigma|_{x}$ is a complex linear isomorphism from $F|_{x}$ to $W$, $\sigma$ is  holomorphic in the vertical directions, i.e., $\sigma_{*}(if')=i\sigma_{*}(f')$, where $f' \in V_{f}F=F|_{x}.$

2) As for the horizontal directions, we need to show that $\sigma_{*}(\mathcal{J}^{(\nabla,I)}v^{\nabla})= \mathcal{I}\sigma_{*}(v^{\nabla})$, where $v^{\nabla}$ is the horizontal lift of $v\in T_{x}M$ to $H_{f}^{\nabla}F \subset T_{f}F$  and $\mathcal{I}$ is the complex structure on $U \times W$.
 Let us first consider,
 \begin{align*}
 \sigma_{*}(\mathcal{J}^{(\nabla,I)}v^{\nabla}) &= \sigma_{*}((Iv)^{\nabla})\\ & = \frac{d}{dt}|_{t=0}\sigma(f(t)),
  \end{align*}
 where $f(t)$ is the parallel translate of $f$ along a curve $\gamma: \mathbb{R} \longrightarrow M$ that satisfies $\gamma(0)=x$ and $\gamma'(0)=Iv$. If we let $f(t)= a_{j}(t)f_{j}|_{\gamma(t)}$ then the above equals
 \[(Iv, \frac{da_{j}(t)}{dt}|_{t=0}w_{j}).\]
Similarly,
      $\sigma_{*}(v^{\nabla}) =(v, \frac{d\tilde{a}_{j}(t)}{dt}|_{t=0}w_{j}),$ where $\tilde{f}(t)= \tilde{a}_{j}(t)f_{j}|_{\tilde{\gamma(t)}}$ is the parallel translate of $f$ along a curve $\tilde{\gamma}: \mathbb{R} \longrightarrow M$ that satisfies $\tilde{\gamma}(0)=x$, $\tilde{\gamma}'(0)=v$. 
      Now since $\mathcal{I}\sigma_{*}(v^{\nabla})=(Iv, i\frac{d\tilde{a}_{j}(t)}{dt}|_{t=0}w_{j}),$ $\sigma$ is pseudoholomorphic if and only if $\frac{da_{j}(t)}{dt}|_{t=0}= i\frac{d\tilde{a}_{j}(t)}{dt}|_{t=0}.$
      
To show this equality, note that the condition $\tilde{\gamma}^{*}{\nabla}\tilde{f(t)}=0$ together with $ a_{j}:=\tilde{a}_{j}(0)=a_{j}(0)$ imply that $i\frac{d\tilde{a}_{j}(t)}{dt}|_{t=0}f_{j}= -ia_{j}\nabla_{v}f_{j}$. This then equals $-a_{j}\nabla_{Iv}f_{j}$ because $\nabla^{0,1}f_{j}=0$, which in turn equals $\frac{da_{j}(t)}{dt}|_{t=0}f_{j}$ since $\gamma^{*}{\nabla}{f(t)}=0$. Hence $\sigma$ is pseudoholomorphic.
\end{proof}

As for the Grassmannians, we have: 
\begin{prop}
The complex structure on $Gr_{k}(F)$ that is induced by $(F,\overline{\partial})$ is $\mathcal{J}^{(\nabla,I)}$.
\end{prop}

The proof of the above proposition and hence of Proposition \ref{PropCSG} is just a straightforward generalization of the previous proof. This then completes the proof of Theorem \ref{ThmGEM} as well.

\subsection{Corollaries of the Embedding}
\label{SecCE}
We will now demonstrate some corollaries of the holomorphic embedding $ \psi: (\mathcal{C}, \mathcal{J}^{(\nabla,I)}) \longrightarrow Gr_{n}(E_{\mathbb{C}}),$ as given in Theorem \ref{ThmGEM}. In particular, we will address certain issues regarding the holomorphic structure of twistor spaces that were raised in Section \ref{SecDOP}.  
 
Let $E$ and $E'$ be two real vector bundles of even rank that are fibered over $(M,I)$ and that are respectively equipped with connections $\nabla$ and $\nabla'$ of (1,1) curvature. 
\begin{prop}
\label{PropISG}
Let $A:E \longrightarrow E'$ be a bundle map such that its $\mathbb{C}$-extension, 
$A: (E_{\mathbb{C}},\nabla^{0,1}) \longrightarrow (E'_{\mathbb{C}},\nabla'^{0,1})$ is an isomorphism of holomorphic vector bundles. Then this map induces a fiber preserving biholomorphism between $(\mathcal{C}(E),\mathcal{J}^{(\nabla,I)})$ and $(\mathcal{C}(E'),\mathcal{J}^{(\nabla',I)})$.
\end{prop}

\begin{proof}
The isomorphism $A: (E_{\mathbb{C}},\nabla^{0,1}) \longrightarrow (E'_{\mathbb{C}},\nabla'^{0,1})$ induces the biholomorphism $\tilde{A}: Gr_{n}(E_{\mathbb{C}}) \longrightarrow Gr_{n}(E'_{\mathbb{C}})$ that is defined by $\tilde{A} (<e_{1},...,e_{n}>_{\mathbb{C}})= <Ae_{1},...,Ae_{n}>_{\mathbb{C}}$. Since $A$ is a real map, $\tilde{A}$ restricts to a biholomorphism between the set $\{P \in Gr_{n}(E_{\mathbb{C}}) | P \oplus \overline{P}=E_{\mathbb{C}}|_{\pi_{Gr(P)}}\}$ in $Gr_{n}(E_{\mathbb{C}})$ and the corresponding one in $Gr_{n}(E'_{\mathbb{C}})$. The proposition then follows from Theorem \ref{ThmGEM} and Proposition \ref{PropVECEM}, which show that these sets are respectively biholomorphic to $(\mathcal{C}(E),\mathcal{J}^{(\nabla,I)})$ and $(\mathcal{C}(E'),\mathcal{J}^{(\nabla',I)})$.
\end{proof}

Now suppose that $E$ and $E'$ are also equipped with respective fiberwise metrics $g$ and $g'$ and that the above connections preserve the appropriate metrics. If we $\mathbb{C}$-bilinearly extend $g$ and $g'$ to $E_{\mathbb{C}}$ and $E'_{\mathbb{C}}$, we then have  

\begin{prop}
\label{PropIS}
Let $A:(E_{\mathbb{C}},\nabla^{0,1}) \longrightarrow (E'_{\mathbb{C}},\nabla'^{0,1})$ be an isomorphism of holomorphic vector bundles that is orthogonal with respect to $g$ and $g'$. Then $A$ induces a fiber preserving biholomorphism between $(\mathcal{T}(E,g),\mathcal{J}^{(\nabla,I)})$ and $(\mathcal{T}(E',g'),\mathcal{J}^{(\nabla',I)})$.
\end{prop}

\begin{proof}
Similar to the proof of Proposition \ref{PropISG}, the isomorphism $A: (E_{\mathbb{C}},\nabla^{0,1}) \\ \longrightarrow (E'_{\mathbb{C}},\nabla'^{0,1})$ induces a biholomorphism $\tilde{A}: Gr_{n}(E_{\mathbb{C}}) \longrightarrow Gr_{n}(E'_{\mathbb{C}})$. Since $A$ is an orthogonal map, $\tilde{A}$ maps the space of maximal isotropics, $MI(E_{\mathbb{C}})$, in $Gr_{n}(E_{\mathbb{C}})$ to the one in $Gr_{n}(E'_{\mathbb{C}})$. The proposition then follows from Proposition \ref{PropGEM}, which shows that $(\mathcal{T}(E,g),\mathcal{J}^{(\nabla,I)})$ and $(\mathcal{T}(E',g'),\mathcal{J}^{(\nabla',I)})$ are respectively biholomorphic to $MI(E_{\mathbb{C}})$ and $MI(E'_{\mathbb{C}})$.
\end{proof}

In the following two sections we consider some applications of the above propositions.

\subsubsection{Cohomology Independence}
\label{SecCD}
Let $(E,g,J) \longrightarrow (M,I)$ be a holomorphic Hermitian bundle fibered over a complex manifold and let $\overline{\partial}$ be the standard $\overline{\partial}$-operator on $E^{1,0} \oplus E^{*1,0}$, where $E^{1,0}$ is the $+i$ eigenbundle of $J$. If we choose $D \in \Gamma(T^{*0,1} \otimes \wedge^{2}E^{*1,0})$ to satisfy $\overline{\partial}D=0$ then, as described in Section \ref{SecDOP}, $\nabla^{Ch(0,1)} + g^{-1}D$ is a $\overline{\partial}$-operator on $E_{\mathbb{C}}=E^{1,0} \oplus E^{0,1}$ and, for $\nabla=\nabla^{Ch}+ g^{-1}D + \overline{g^{-1}D}$, the twistor space $(\mathcal{T}(E), \mathcal{J}^{(\nabla,I)})$ is a complex manifold. 
If we now let $B \in \Gamma(\wedge^{2}E^{*1,0})$ then $\nabla^{Ch(0,1)} + g^{-1}(D+ \overline{\partial}B)$ is another $\overline{\partial}$-operator on $E_{\mathbb{C}}$ and it is natural to wonder, as in Section \ref{SecDOP}, whether the associated twistor space is biholomorphic to the previous one. In other words, does the above give a well defined mapping from the Dolbeault cohomology group $H^{0,1}(\wedge^{2}E^{*1,0})$ to the isomorphism classes of complex structures on $\mathcal{T}$? 

 By using Proposition \ref{PropIS}, we will show here that such a mapping does indeed exist. As a first step, let us consider the section of $O(E_{\mathbb{C}},g)$ $exp(g^{-1}B)$, which equals $(1+ g^{-1}B)$ since $(g^{-1}B)^{2}=0$. We then have

\begin{prop} The map
$exp(-g^{-1}B): (E_{\mathbb{C}}, \nabla^{Ch(0,1)} + g^{-1}D) \longrightarrow (E_{\mathbb{C}},\nabla^{Ch(0,1)} + g^{-1}(D+ \overline{\partial}B))$ is an isomorphism of holomorphic vector bundles. 
\end{prop}
\begin{proof}
Let $(\nabla^{Ch(0,1)} + g^{-1}D)v=0$ and consider 
\begin{align*}
& (\nabla^{Ch(0,1)} + g^{-1}(D+ \overline{\partial}B))(1- g^{-1}B)v \\
&= -\nabla^{Ch(0,1)}(g^{-1}Bv) +(g^{-1}\overline{\partial}B)v \\
&=-(\nabla^{Ch(0,1)}g^{-1}B)v - g^{-1}B\nabla^{Ch(0,1)}v + (g^{-1}\overline{\partial}B)v.
\end{align*}
Since the first and last terms cancel, we are left with $-g^{-1}B\nabla^{Ch(0,1)}v = -g^{-1}B(-g^{-1}Dv)=0$. This then proves the proposition.
\end{proof}

By Proposition \ref{PropIS}, we can now conclude that the twistor spaces mentioned above are biholomorphic:

\begin{prop} $exp(-g^{-1}B)$ induces a fiber preserving  biholomorphism between 
$(\mathcal{T}, \mathcal{J}^{(\nabla,I)})$ and 
$(\mathcal{T}, \mathcal{J}^{(\nabla',I)}),$ where 
$\nabla^{0,1}=\nabla^{Ch(0,1)} + g^{-1}D$ and $\nabla'^{0,1}=\nabla^{Ch(0,1)} + g^{-1}(D+ \overline{\partial}B)$.
\end{prop}

As a corollary, we have

\begin{prop}
The map $[D] \longrightarrow [\mathcal{J}^{(\nabla,I)}]$, where $\nabla^{0,1}=\nabla^{Ch(0,1)} + g^{-1}D$, from the Dolbeault cohomology group $H^{0,1}(\wedge^{2}E^{*1,0})$ to the isomorphism classes of complex structures on $\mathcal{T}(E,g)$ is well defined. 
\end{prop}

\subsubsection{Changing the Metric}
\label{SecCTM}
In the previous example we worked with a fixed metric $g$; but what if we were to choose another metric $g'$ on $E$ that is compatible with $J$---then is it true that $(\mathcal{T}(g), \mathcal{J}^{(\nabla^{Ch},I)})$ and $(\mathcal{T}(g'), \mathcal{J}^{(\nabla^{Ch'},I)})$ are biholomorphic? This is part of a more general question that was posed in Section \ref{SecDOP}: in that section we used a fixed metric, $g$, to define $\overline{\partial}$-operators on $E_{\mathbb{C}}$ and thus complex structures on $\mathcal{T}(g)$---but if we were to choose another metric $g'$ then do we obtain new complex manifolds by considering $\mathcal{T}(g')$?

To address these questions, let us first recall some of the details of that section. 
Let $(E,J) \longrightarrow (M,I)$ be a holomorphic vector bundle, considered as a real bundle with fiberwise complex structure $J$, that is fibered over a complex manifold. Defining $<,>$ and $\overline{\partial}$ to be the standard inner product and $\overline{\partial}$-operator on $E^{1,0} \oplus E^{*1,0}$, let us consider the $\overline{\partial}$-operator $\bar{\partial} + \mathcal{D}'^{0,1}$, where $\mathcal{D}'^{0,1} \in \Gamma(T^{*0,1} \otimes \mathfrak{so}(E^{1,0} \oplus E^{*1,0})).$ If $g$ is a fiberwise metric on $E$ that is compatible with $J$ then, as in Section \ref{SecDOP}, we can use the orthogonal isomorphism \[1 \oplus g: (E_{\mathbb{C}}= E^{1,0} \oplus E^{0,1}, \frac{g}{2}) \longrightarrow (E^{1,0} \oplus E^{*1,0}, <,>)\]
to obtain the $\overline{\partial}$-operator $\nabla^{Ch(0,1)} + \mathcal{D}_{g}^{0,1}$ on $E_{\mathbb{C}}$ as well as the complex structure $\mathcal{J}^{(\nabla^{Ch}+\mathcal{D}_{g}, I)}$ on $\mathcal{T}(g)$. (Here, $\mathcal{D}_{g}= \mathcal{D}_{g}^{0,1} +\overline{\mathcal{D}_{g}^{0,1}}$.) 

Similarly, if $g'$ is another fiberwise metric that is compatible with $J$ then we have the complex structure $\mathcal{J}^{(\nabla^{Ch'}+\mathcal{D}_{g'}, I)}$ on $\mathcal{T}(g')$. The goal then is to use Proposition \ref{PropIS} to show that the complex manifolds $\mathcal{T}(g)$ and $\mathcal{T}(g')$ are equivalent under a fiberwise biholomorphism.

First note, that if we compose the map $(1 \oplus g)$ with $(1 \oplus g')^{-1}$ then we obtain the following isomorphism of holomorphic vector bundles: 
\begin{align*}
 (E_{\mathbb{C}},\nabla^{Ch(0,1)}+\mathcal{D}^{0,1}_{g}&) \longrightarrow (E_{\mathbb{C}},\nabla^{Ch'(0,1)}+\mathcal{D}^{0,1}_{g'})\\
v^{1,0}+v^{0,1} & \longrightarrow (v^{1,0} + g'^{-1}gv^{0,1}),
\end{align*} 
where we have used the decomposition, $E_{\mathbb{C}}= E^{1,0} \oplus E^{0,1}$. As this is an orthogonal map from $(E_{\mathbb{C}},g)$ to $(E_{\mathbb{C}},g')$, by Proposition \ref{PropIS} we have

\begin{prop}
There exists a fiber preserving biholomorphism between $(\mathcal{T}(g),\mathcal{J}^{(\nabla^{Ch}+\mathcal{D}_{g}, I)})$ and $(\mathcal{T}(g'),\mathcal{J}^{(\nabla^{Ch'}+\mathcal{D}_{g'}, I)})$. 
\end{prop}

In particular, if we set $\mathcal{D}'^{(0,1)}$ to zero, we have:  
\begin{prop}Let $(E,J) \longrightarrow (M,I)$ be a holomorphic vector bundle that is equipped with two Hermitian metrics $g$ and $g'$. Then $(\mathcal{T}(g),\mathcal{J}^{(\nabla^{Ch},I)})$ and $(\mathcal{T}(g'),\mathcal{J}^{(\nabla^{Ch'}, I)})$ are biholomorphic. 
\end{prop}

\section{Acknowledgments}
I would like to thank Blaine Lawson and Nigel Hitchin for helpful discussions. I would also like to thank Alexander Kirillov and Martin Ro\v{c}ek.

\textsc{Department of Mathematics, UC Riverside, Riverside, CA 92521} \\

\textit{E-mail Address:} \texttt{gindis@ucr.edu} 

\end{large}
\end{document}